\documentclass[11pt]{article}

\usepackage[T1]{fontenc}
\usepackage[utf8]{inputenc}

\usepackage{amsmath,amssymb,amsfonts,amsthm}
\usepackage{graphicx}

\usepackage{tikz}
\usepackage{pgfplots}
\pgfplotsset{compat=1.18}

\usepackage{booktabs}
\usepackage{enumitem}
\usepackage{natbib}

\usepackage{hyperref}
\hypersetup{colorlinks=true, linkcolor=blue, citecolor=blue, urlcolor=blue} 
\usepackage[nameinlink]{cleveref}

\usepackage[margin=1in]{geometry}

\usepackage{setspace}
\theoremstyle{plain}
\newtheorem{theorem}{Theorem}[section]
\newtheorem{lemma}[theorem]{Lemma}

\newtheorem{proposition}[theorem]{Proposition}

\theoremstyle{definition}
\newtheorem{definition}[theorem]{Definition}
\newtheorem{assumption}[theorem]{Assumption}
\newtheorem{example}[theorem]{Example}

\theoremstyle{remark}
\newtheorem{remark}[theorem]{Remark}

\newcommand{\RR}{\mathbb{R}}
\newcommand{\EE}{\mathbb{E}}
\newcommand{\PP}{\mathbb{P}}

\DeclareMathOperator{\re}{Re}
\DeclareMathOperator{\tr}{tr}

\title{Curvature-Dependent Path Concentration in Stochastic Fast-Slow Systems with Noise on the Slow Variable}

\author{
  Yefan Wu\thanks{Email: \texttt{wuyefan718@gmail.com}, \texttt{yewu0336@uni.sydney.edu.au}}\\[0.2cm]
  School of Mathematics and Statistics, University of Sydney
}

\date{}

\begin{document}

\maketitle

\begin{abstract}
We study stochastic fast-slow systems in which the noise acts
exclusively on the slow variable:
$dx = f(x,y)\,dt$,
$dy = \varepsilon\,g(x,y)\,dt + \sigma\,h(y)\,dW_t$.
While the path-concentration theory for noise on the fast variable
is well developed, the complementary case of noise only on the slow variable has remained largely unexplored, with a recent exception treating the fold bifurcation. For general, uniformly normally-hyperbolic deterministic slow manifolds $x = X^*(y)$,
we derive rigorous pathwise estimates showing that the deviation
$z = x - X^*(y)$ concentrates with exponential tail bounds
over the slow timescale $[0,T/\varepsilon]$. A central finding is
that the It\^o correction arising from the curvature $D^2X^*$ of
the slow manifold introduces a systematic $O(\sigma^2\|D^2X^*\|)$
bias that enters the concentration bound alongside the transmitted
noise term of scale $\sigma/\sqrt{\lambda_0}$,  where $\lambda_0 > 0$
denotes the uniform spectral gap of the linearized fast dynamics
along the manifold. We identify a
geometric critical noise scale
$\sigma_c(\varepsilon) = C_0\min\!\bigl(\sqrt{\varepsilon},\,
\varepsilon^{1/4} L_{\mathrm{geom}}/\sqrt{\lambda_0}\bigr)$,
where $L_{\mathrm{geom}} = \sqrt{\lambda_0/(\|D^2X^*\|\|h\|^2)}$
is a local geometric scale of the manifold. Under this sufficient
noise restriction, we give an explicit geometry-dependent confidence
radius. For fixed geometric parameters and any fixed $p>0$, its
order on $[0,T/\varepsilon]$ is
$O(\varepsilon+\sigma^2+\sigma\sqrt{\log(1/\varepsilon)})$,
with stopped fast-variable failure probability at most $\varepsilon^p$;
slow-variable exit is controlled separately by $e^{-C_K}$.
We also prove a slow-variable adiabatic error bound of order
$\varepsilon+\sigma\sqrt{\varepsilon}+\sigma^2\|D^2X^*\|$
for fixed geometry, with the fluctuation and confinement
probabilities specified explicitly.
Thus the curvature correction enters the fast-variable concentration
bound, while its contribution to the slow-error upper bound is
$O(\varepsilon)$ in the fixed-geometry regime.
\end{abstract}


\section{Introduction}

\subsection{Stochastic fast-slow systems: state of the art}
Fast-slow systems with time-scale separation $\varepsilon \ll 1$
form a central paradigm in applied dynamical systems
\citep{Kuehn2015}. In the deterministic setting, Geometric
Singular Perturbation Theory (GSPT), initiated by \citet{Fenichel1979} and systematised by \citet{Jones1995} and \citet{Kaper1999}, establishes that a normally-hyperbolic slow manifold $M_0 = \{(X^*(y),y)\}$ persists for $\varepsilon > 0$ as a nearby invariant manifold $M_\varepsilon$ that attracts trajectories exponentially on the fast timescale
(and by a factor of order $e^{-\lambda_0 s/\varepsilon}$
over a fixed slow-time duration $s>0$).

Here $\lambda_0 > 0$ denotes the uniform spectral gap of the
linearised fast dynamics along $M_0$; its precise definition
is given in~\eqref{eq:lambda}.

When the system is stochastically perturbed, the pathwise
behaviour depends crucially on which variables are noisy.
The most extensively studied setting is noise on the
fast variable (or on both variables), where \citet{BerglundGentz2003, BerglundGentz2006} developed a sample-paths theory showing that trajectories concentrate in a tube of width $\sigma/\sqrt{\lambda_0}$ around $M_0$, with exponential tail bounds valid up to Kramers' escape times. Near bifurcation points, they uncovered critical noise scales
$\sigma_c \sim \varepsilon^{1/4}$ at which the tube
diverges and early-jump mechanisms dominate
\citep{BerglundGentzKuehn2012}.

Parallel to the sample-paths approach, random invariant 
manifold theory \citep{Arnold1998, DuanLuSchmalfuss2003, SchmalfussSchneider2008} 
constructs a random slow manifold as a random graph over the 
slow-variable space, and normal-form methods 
\citep{Roberts2008, WangRoberts2013} yield systematic reduced models. 
Stochastic averaging principles \citep{Khasminskii1968, Veretennikov1999, Veretennikov2000, VandenEijnden2003, PavliotisStuart2008} provide 
complementary convergence rates and identify regimes where 
adiabatic reduction breaks down. Computational and 
model-reduction frameworks for such systems are reviewed 
by \citet{GivonKupfermanStuart2004}.

A crucial structural feature shared by all of the above
frameworks is that the noise acts on (or is coupled to)
the fast variable. In the classical averaging of
\citet{Khasminskii1968}, noise enters the fast equation, and the
averaging rates depend on~$\varepsilon$ alone; no geometric
property of the slow manifold enters. In the sample-paths theory
of \citet{BerglundGentz2003,BerglundGentz2006},
noise directly perturbs the fast variable, producing a tube of
width $\sigma/\sqrt{|\lambda|}$ around the manifold whose bound
depends only on the spectral gap, not on the manifold curvature.
Similarly, the normal-form reductions of \citet{Roberts2008}
and the averaging schemes of \citet{GivonKupfermanStuart2004} treat
noise configurations in which the It\^o curvature correction
$-\frac{\sigma^2}{2}D^2X^*[h\otimes h]$ does not arise.
The curvature mechanism identified in the present work is a
genuine consequence of the noise acting exclusively on the
slow variable.

\subsection{The gap: noise on the slow variable}
In all of the above work, the noise acts on the fast
variable, on both variables, or on the system in a
symmetric fashion. The practically and mathematically
natural case in which noise acts only on the slow
variable, corresponding to a deterministic fast
subsystem driven by a randomly perturbed slow parameter, has remained almost entirely unexplored. The sole recent exception is \citet{BergeotBerglundZogheib2025}, 
who study the effect of noise on the slow variable near a 
fold/saddle-node bifurcation of the slow manifold, showing that 
noise induces an $O(\sigma^2)$ systematic drift in the slow 
variable after the fold, expressible in terms of Airy 
functions. Their analysis is, however, restricted to the
vicinity of the bifurcation point and does not address
general normally-hyperbolic slow manifolds.

The difficulty is structural. When noise acts on the slow
variable $y$, it enters the fast-fibre deviation $z = x - X^*(y)$
through two distinct mechanisms: (i) the Jacobian
$DX^*(y)$ propagates slow-variable noise into the fast
fibre as a diffusion term, and (ii) applying It\^o's formula to $X^*(y)$ yields a quadratic 
variation term that manifests in the $z$-equation as a systematic 
curvature-induced drift $-\tfrac{\sigma^2}{2} D^2X^*(y)[h(y)\otimes h(y)]\,dt$, 
which is absent when noise acts on the fast variable. Mechanism (ii) 
is novel and depends on the second derivative of the slow 
manifold, that is, on its extrinsic curvature as an embedded 
submanifold of the phase space. (In coordinates: 
$(D^2X^*[h\otimes h])_i = \sum_{j,k=1}^m \sum_{\ell=1}^d 
\frac{\partial^2 X^*_i}{\partial y_j \partial y_k}\,h_{j\ell}h_{k\ell}$.)

\subsection{Contributions}
The norm conventions used below are specified in
Section~\ref{sec:manifold_norms}.

The present paper provides a rigorous path-concentration
theory for stochastic fast-slow systems with noise only on
the slow variable, valid for general uniformly
normally-hyperbolic deterministic slow manifolds.
Our contributions are:

\begin{enumerate}[label=(\roman*),leftmargin=*]
\item We derive the exact SDE for the fast-fibre deviation
      $z$ and identify the It\^o curvature term
      $b_\sigma = -\frac{\sigma^2}{2}D^2X^*[h\otimes h]$
      as the central new mechanism.

\item We construct a parameter-dependent Lyapunov function
      $V = z^T P(y)z$ and prove a Foster--Lyapunov bound
      whose forcing constant $D_0$ depends explicitly on
      $\|D^2X^*\|$, establishing a geometric dependence in
      the concentration bound.

\item Using exponential Lyapunov and stopping-time estimates, we
      prove a maximal tail bound with explicit dependence on the
      radius, noise and observation time. For the confidence radius
      $r_\rho$ and $\rho=\varepsilon^p$, the joint failure probability
      of concentration and slow-variable confinement is at most
      $\varepsilon^p+e^{-C_K}$ on $[0,T/\varepsilon]$.
      We establish the curvature-dependent sufficient noise scale
      \[
        \sigma_c(\varepsilon)
        = C_0\min\!\bigl(\sqrt{\varepsilon},\,
          \varepsilon^{1/4}L_{\mathrm{geom}}/\sqrt{\lambda_0}\bigr),
      \]
      where $L_{\mathrm{geom}}=\sqrt{\lambda_0/(M_{C_2}M_h^2)}$,
      and verify its energy self-consistency explicitly.

\item We prove that the adiabatic slow-variable error is
      $O(\varepsilon + \sigma\sqrt{\varepsilon} +
      \sigma^2\|D^2X^*\|)$, and show that the curvature
      contribution is $O(\varepsilon)$ under $\sigma\le\sigma_c$
      when the geometric parameters are fixed.
\end{enumerate}

The proof technique shares the Foster--Lyapunov and
exponential-estimate architecture of the classical
Berglund--Gentz framework, but two structural novelties distinguish
the present analysis. First, noise arrives at the fast
variable indirectly. In the classical setting, noise acts directly
on $x$, producing a dominant $O(\sigma^2)$ diffusive term in the
$z$-equation. Here, noise enters $z$ only through the Jacobian
coupling $DX^*$ (generating $-\sigma DX^* h\,dW_t$, of order
$\sigma$) and through the It\^o
curvature correction $-\tfrac12\sigma^2 D^2X^*[h\otimes h]\,dt$
(which is a drift term, not a diffusion term). This indirect
transmission produces cross-It\^o terms in the generator of
$V = z^T P(y)z$ (Term~VII in Lemma~\ref{lem:ito_V}) that have no
classical analogue. Second, the Lyapunov matrix must vary with the slow parameter. Because the linearised fast drift $A(y):=D_xf(X^*(y),y)$ varies 
along the trajectory, $P(y)$ must solve the parameter-dependent 
Lyapunov equation $A(y)^T P(y) + P(y)A(y) = -I$, introducing $DP$ 
and $D^2P$ terms into the It\^o expansion of $V$ that are absent 
when $A$ is constant. Bounding $\|DP\|$ and $\|D^2P\|$ via 
differentiated Sylvester equations (Appendix~\ref{sec:appendix}) 
is essential; the constants acquire higher powers of $K_A$ (notably 
$K_A^4$ in $M_{DP}$), which makes the effective decay rate 
$\tilde\lambda$ more sensitive to both $\varepsilon$ and $\sigma$ 
while the diffusion energy and squared curvature forcing
lead to the two restrictions verified in~\eqref{eq:threshold_energy_control}. In the 
illustrative Example~5.1 below, $A(y)$ is constant, so that $DP$ 
and $D^2P$ vanish and the cross-It\^o corrections are absent. The 
full proof, however, tracks these terms explicitly under the general 
assumptions of Section~2, and a systematic treatment of parameter 
families in which both $\|D^2X^*\|$ and $\|DA\|$ grow simultaneously 
is a natural direction for future work.

We emphasise the explicit curvature dependence of the fast-variable
concentration bound. For fixed geometry, the curvature contribution
to the slow-error upper bound remains $O(\varepsilon)$. For systems with
$\|D^2X^*\| = O(1)$, the classical threshold
$\sigma \ll \sqrt{\varepsilon}$ dominates and curvature
does not tighten it. The geometric scale becomes binding
precisely when the manifold is sharply curved ($\|D^2X^*\| \gg 1$), which occurs near (but not at)
loss-of-normal-hyperbolicity points.
This curvature-induced refinement identifies the graph
curvature of the slow manifold, quantified by $\|D^2X^*\|$,
as an independent control parameter in the concentration
theory of stochastic fast-slow systems with noise acting on
the slow variables. Even when the curvature branch is not
active, the explicit dependence of the concentration bound
on $\|D^2X^*\|$ provides a systematic geometric correction
to classical estimates based solely on the spectral gap.
This makes our result complementary to the Berglund--Gentz
near-bifurcation theory: their analysis captures
$\lambda_0 \to 0$; ours captures $\|D^2X^*\| \to \infty$
with $\lambda_0 > 0$ fixed.

\subsection{Paper outline}
Section~\ref{sec:setup} states the system, the assumptions and the
Lyapunov construction. Section~\ref{sec:main} presents the main
theorem. Section~\ref{sec:proofs} gives the complete proofs.
Section~\ref{sec:discussion} discusses the phase diagram,
limitations and physical motivations.
Appendix~\ref{sec:appendix} contains technical details on
the parameter-dependent Lyapunov matrix.

\section{Setup and Assumptions}
\label{sec:setup}

\subsection{The system}
Consider the $n$-dimensional fast variable $x \in \RR^n$
and $m$-dimensional slow variable $y \in \RR^m$ governed by
\begin{subequations}\label{eq:system}
\begin{align}
dx &= f(x,y)\,dt, \label{eq:x}\\
dy &= \varepsilon\,g(x,y)\,dt + \sigma\,h(y)\,dW_t, \label{eq:y}
\end{align}
\end{subequations}
where $W_t$ is a standard $d$-dimensional Brownian motion,
$\varepsilon \ll 1$ is the time-scale separation parameter,
and $\sigma > 0$ is the noise intensity. The noise coefficient $h(y)$ takes values in $\RR^{m \times d}$.

We assume that the fast equation contains no direct noise
and that the noise coefficient $h$ is independent of $x$;
we discuss the relaxation of the latter assumption
in Section~\ref{sec:discussion}.

\subsection{Deterministic slow manifold}
\label{sec:manifold_norms}
Assume there exists a function $X^* \in C^3(K,\RR^n)$,
defined on a compact set $K \subset \RR^m$, such that
\begin{equation}\label{eq:manifold}
f(X^*(y),y) = 0 \quad \text{for all } y \in K.
\end{equation}
The graph $M_0 = \{(X^*(y),y) : y \in K\}$ is the
deterministic slow manifold.
For the adiabatic approximation below, fix a smooth extension of
$X^*(\cdot)$ to $\RR^m$ with bounded first derivative, using a
cutoff outside a neighbourhood of $K$. The manifold geometry
and its uniform bounds are evaluated on $K$; the extended graph
defines the reduced equation outside $K$.

\paragraph{Norm conventions.}
For vectors $v \in \RR^n$, $|v|$ denotes the Euclidean norm.
For matrices $A \in \RR^{n \times n}$, $\|A\|$ denotes the spectral
norm (induced operator norm). For the second derivative
$D^2X^*(y) \in \RR^{n \times m \times m}$, viewed as a third-order
tensor, we use the induced norm
\[
\|D^2X^*(y)\| = \sup_{|u|=1,\,|v|=1}|D^2X^*(y)[u,v]|,
\]
where
$D^2X^*(y)[u,v] \in \RR^n$ is the vector obtained by contracting
the tensor with $u, v \in \RR^m$. This is the standard operator norm
for bilinear maps induced by the Euclidean norm. In particular,
$|D^2X^*(y)[u,v]| \le \|D^2X^*(y)\|\,|u|\,|v|$ for all
$u,v\in\RR^m$. With this convention,
\[
|D^2X^*(y)[h(y)\otimes h(y)]|
\;\le\; \|D^2X^*(y)\|\,|h(y)|^2.
\]
For the noise coefficient matrix $h(y) \in \RR^{m \times d}$, we
interpret $|h(y)|$ as the Frobenius norm
$\|h(y)\|_F = \sqrt{\tr(h(y)^T h(y))}$, so that
$|h(y)|^2 = \sum_{j=1}^m\sum_{\ell=1}^d h_{j\ell}(y)^2$ and the
inequality above holds with the induced tensor norm on $D^2X^*$.

Define the fast stability matrix, its uniform spectral bound
and the Jacobian and curvature of the slow manifold:
\begin{align}
A(y) &= D_x f(X^*(y),y) \in \RR^{n \times n}, \label{eq:A}\\
\lambda_0 &= \inf_{y \in K}\min_{i=1,\dots,n}\bigl(-\re\lambda_i(A(y))\bigr) > 0, \label{eq:lambda}\\
M_C &= \sup_{y \in K}\|DX^*(y)\|, \label{eq:MC}\\
M_{C_2} &= \sup_{y \in K}\|D^2X^*(y)\|. \label{eq:MC2}
\end{align}

\begin{assumption}[Uniform normal hyperbolicity]\label{ass:NH}
Assume the spectral condition~\eqref{eq:lambda} holds.
There exists a constant $K_A \ge 1$ such that
\[
\|e^{A(y)t}\| \le K_A e^{-\lambda_0 t}
\quad \text{for all } t \ge 0,\; y \in K.
\]
\end{assumption}

We state Assumption~\ref{ass:NH} explicitly to track the
uniform prefactor $K_A$, which controls transient amplification.
The spectral condition~\eqref{eq:lambda} guarantees a uniform
semigroup estimate with any decay rate strictly below
$\lambda_0$, but does not in general guarantee the endpoint
rate $\lambda_0$ itself, owing to possible nontrivial Jordan
blocks at the spectral boundary.

\subsection{Constants and regularity}
\begin{definition}[System constants]\label{def:constants}
Let $U$ be a neighbourhood of $M_0$ over $K$. We define:
\begin{align*}
&M_A = \sup_K \|A(y)\|, \quad
M_g = \sup_U \|g\|, \quad
M_{D_xg} = \sup_U \|D_x g\|,\\
&M_h = \sup_K \|h(y)\|_F, \quad
M_B = \sup_K \|D_x^2 f(X^*,y)\|, \quad
M_{DF} = \sup_K \|DF(y)\|,
\end{align*}
where $F(y) = g(X^*(y),y)$ and $DF(y) = D_xg \cdot DX^* + D_yg$.
All these constants are assumed to be finite.
\end{definition}

\begin{assumption}[Smoothness]\label{ass:smooth}
$f \in C^4$, $g \in C^2$, $h \in C^2$, and $X^* \in C^3(K,\RR^n)$.
The assumption $f \in C^4$ is required for the boundedness of
$D^2P(y)$ in Appendix~\ref{sec:appendix}.
\end{assumption}

\begin{assumption}[Lipschitz and non-degeneracy]\label{ass:globs}
$g$ and $h$ are globally Lipschitz with constants $L_g, L_h$.
There exists $h_{\min} > 0$ such that $\|h(y)\| \ge h_{\min}$
for all $y \in K$.
\end{assumption}

\begin{assumption}[Confinement of slow dynamics]\label{ass:confinement}
There exists a $C^2$ function $\psi\colon K \to [0,\infty)$
satisfying:
\begin{enumerate}[label=(\roman*)]
\item $K = \{\psi \le R_K\}$ is an exact sublevel set for some $R_K > 0$;
\item there exist a compact ``safe core''
      $K_0 = \{\psi \le R_\star\} \subset \mathrm{int}(K)$ with
      $0 \le R_\star < R_K$, containing $y_0$, and a constant
      $c_\psi > 0$ such that for all
      $y \in K \setminus \mathrm{int}(K_0)$ and all $x$
      with $|x - X^*(y)| \le R_0$ (where $R_0$ is from
      Assumption~\ref{ass:stop}),
      $\nabla\psi(y) \cdot g(x,y) \le -c_\psi$;
\item $M_\psi = \sup_{K}\|\nabla\psi\|$ and
      $M_{D^2\psi} = \sup_{K}\|D^2\psi\|$ are finite,
      and $M_\psi > 0$.
\end{enumerate}
\end{assumption}

\begin{remark}[Physical meaning of confinement]
Condition~(ii) requires the deterministic slow flow to point
inward across every level set of $\psi$ in the annular region
$K \setminus \mathrm{int}(K_0)$, whose $\psi$-width
is $\delta_\psi := R_K - R_\star > 0$. This is satisfied, for example,
when the slow flow has a globally attracting equilibrium $y^*$
inside $K_0$ and $K$ is a sublevel set
of a Lyapunov function for the reduced dynamics
$\dot{y} = g(X^*(y),y)$.
\end{remark}

\subsection{Parameter-dependent Lyapunov matrix}
For each $y \in K$, let $P(y) \in \mathrm{Sym}^+_n$ be the
unique solution of the Lyapunov equation
\begin{equation}\label{eq:lyap}
A(y)^T P(y) + P(y)\,A(y) = -I_n.
\end{equation}

The representation
\[
P(y)=\int_0^\infty e^{A(y)^Ts}e^{A(y)s}\,ds
\]
yields the uniform comparison bounds
\begin{equation}\label{eq:Pbounds}
c_1 I\le P(y)\le c_2 I,\qquad
c_1=\frac{1}{2M_A},\qquad
c_2=\frac{K_A^2}{2\lambda_0}.
\end{equation}
The upper bound follows from Assumption~\ref{ass:NH}.
For the lower bound, $\|e^{-A(y)s}\|\le e^{M_As}$ implies
\[
|e^{A(y)s}v|\ge e^{-M_As}|v|
\qquad\text{for every }v\in\RR^n.
\]
Consequently,
\[
v^TP(y)v
\ge\int_0^\infty e^{-2M_As}|v|^2\,ds
=\frac{|v|^2}{2M_A}.
\]

We define the bounds $M_{DP} = \sup_K \|DP(y)\|$ and
$M_{D^2P} = \sup_K \|D^2P(y)\|$; explicit expressions are
derived in Appendix~\ref{sec:appendix}.

\begin{assumption}[Stopping radius]\label{ass:stop}
There exists $R_0 > 0$ and constants $C_f, C_g > 0$ such that
on $\{|z| \le R_0\} \cap \{y \in K\}$, the Taylor remainders satisfy
$\|R_f(z,y)\| \le C_f |z|^3$ and $\|R_g(z,y)\| \le C_g|z|^2$, where
$R_f, R_g$ are defined in Lemma~\ref{lem:expansion}.
\end{assumption}

For ease of reference, Table~\ref{tab:constants}
summarises the principal constants used below.

\begin{table}[htbp]
\centering
\small
\caption{Principal constants and their definitions.}
\label{tab:constants}
\begin{tabular}{@{}p{3cm}p{11cm}@{}}
\toprule
Symbol & Meaning and definition \\
\midrule
$\lambda_0$
& Uniform spectral gap; see~\eqref{eq:lambda}. \\

$K_A$
& Uniform semigroup prefactor;
see Assumption~\ref{ass:NH}. \\

$M_C$, $M_{C_2}$
& Uniform bounds on $DX^*$ and $D^2X^*$;
see~\eqref{eq:MC} and~\eqref{eq:MC2}. \\

$M_h$
& Uniform Frobenius-norm bound on $h$;
see Definition~\ref{def:constants}. \\

$c_1$, $c_2$
& Comparison constants for the Lyapunov matrix $P$;
see~\eqref{eq:Pbounds}. \\

$M_{DP}$, $M_{D^2P}$
& Uniform bounds on the derivatives of $P$;
see Appendix~\ref{sec:appendix}. \\

$\gamma_V$, $D_0$
& Decay and forcing coefficients in the drift estimate;
see~\eqref{eq:foster} and~\eqref{eq:D0}. \\

$C_{\mathrm{fast}}$
& Constant setting the fast-fluctuation threshold;
see Lemma~\ref{lem:fast}. \\
\bottomrule
\end{tabular}
\end{table}

\section{Main Results}
\label{sec:main}

We define the local geometric scale of the slow manifold:
\begin{equation}\label{eq:Lgeom}
L_{\mathrm{geom}} \;=\; \sqrt{\frac{\lambda_0}{M_{C_2}\,M_h^2}}.
\end{equation}
This quantity compares the normal attraction rate $\lambda_0$ to
the product of the manifold's curvature $\|D^2X^*\|$ and the noise
coupling $\|h\|^2$. It is large for flat, strongly attracting
manifolds with weak noise, and small for sharply curved, weakly
attracting manifolds with strong noise.

For a fixed slow-time horizon $T>0$, define the geometric
critical noise scale
\begin{equation}\label{eq:sigmac}
\sigma_c(\varepsilon) \;=\; C_0\,\min\!\left(\sqrt{\varepsilon},\;\;
\varepsilon^{1/4}\,\frac{L_{\mathrm{geom}}}{\sqrt{\lambda_0}}\right),
\end{equation}
If $M_{C_2}=0$, the second branch is interpreted as $+\infty$.
Here $C_0 > 0$ is an $O(1)$ constant (as $\varepsilon \to 0$)
depending only on $\lambda_0$, $K_A$, $c_1$, $c_2$, $M_C$, $M_h$,
$M_{DP}$, $R$, $T$, and the confinement data in
Assumption~\ref{ass:confinement}.
Its dependence is tracked through the energy estimate
\eqref{eq:threshold_energy_control}; the two branches provide
a sufficient regime in which $D_0/\varepsilon$ remains bounded.

\begin{remark}[Asymptotic framework]\label{rem:asy}
All constants appearing in $\sigma_c(\varepsilon)$ and in the
probability bounds of Theorem~\ref{thm:main}, namely
$\lambda_0$, $K_A$, $c_1$, $c_2$, $M_C$, $M_{C_2}$, $M_{DP}$,
$M_{D^2P}$, $M_h$, are fixed geometric parameters of the given
system and do not depend on $\varepsilon$ or $\sigma$.
In the single-parameter limit $\varepsilon \to 0$ with these
constants fixed, $C_0$ and the constants in the selected confidence bounds are $O(1)$, and the
classical branch $\sigma_c \sim \sqrt{\varepsilon}$ typically
dominates (since $M_{C_2}\,M_h^2 = O(1)$ implies the curvature
branch is inactive).

The curvature branch becomes binding precisely when one considers
a family of slow manifolds parametrised by an additional
geometric parameter $\delta \to 0$ for which $M_{C_2} = M_{C_2}(\delta)
\to \infty$. In such a two-parameter regime $(\varepsilon, \delta) \to (0,0)$, 
the constants $M_C$ and $M_{C_2}$ scale with $\delta$, which in turn dictates the $\delta$-scaling of the prefactor $C_0 = C_0(\delta)$ and the geometric length scale $L_{\mathrm{geom}} = L_{\mathrm{geom}}(\delta)$. 
This joint limit is analysed explicitly in Example~\ref{ex:sigmoid}.
\end{remark}

\begin{theorem}[Curvature-dependent path concentration and adiabatic error]
\label{thm:main}
Under Assumptions~\ref{ass:NH}, \ref{ass:smooth}, \ref{ass:globs},
\ref{ass:confinement}, and~\ref{ass:stop}, consider the system
\eqref{eq:system} with initial conditions satisfying
$|x_0 - X^*(y_0)| = O(\varepsilon)$ and $y_0 \in K_0
\subset \mathrm{int}(K)$ (where $K_0$ is the safe core from
Assumption~\ref{ass:confinement}).
This initial separation is natural when $x_0$ lies on the
Fenichel slow manifold $M_\varepsilon$, which is $O(\varepsilon)$-close
to $M_0$, or when the system is prepared with the fast variable
equilibrated to the slow variable.
Define the exit time of the slow variable from $K$:
\[
\tau_K = \inf\{t > 0 : y(t) \notin K\}.
\]
For the fixed horizon $T>0$, all $\sigma\le\sigma_c(\varepsilon)$,
and $\varepsilon$ sufficiently small, the following hold:

\medskip\noindent
\textbf{(a) Fast-variable path concentration.}
Fix $p>0$ and $0<\varepsilon<1$.
Set $\rho=\varepsilon^p$ and define
$r_\rho(\varepsilon,\sigma)$ by
\eqref{eq:confidence_radius}.
For fixed system parameters,
$\sigma\le\sigma_c(\varepsilon)$,
and $\varepsilon$ sufficiently small,
$r_\rho<R$ and
\begin{equation}\label{eq:path_conc}
\begin{aligned}
&\PP\!\left(
\sup_{0\le t\le T/\varepsilon}
|x(t)-X^*(y(t))|
\le r_\rho(\varepsilon,\sigma),
\quad \tau_K>T/\varepsilon
\right)\\
&\qquad\ge 1-\varepsilon^p-e^{-C_K},
\end{aligned}
\end{equation}
where $C_K>0$ is supplied by the confinement
estimate~\eqref{eq:exit_bound}.

For fixed geometric parameters and
$|z_0|=O(\varepsilon)$, the radius satisfies
\[
r_\rho(\varepsilon,\sigma)
=O\!\left(
\varepsilon+\sigma^2
+\sigma\sqrt{\log(1/\varepsilon)}
\right).
\]
The full dependence on the geometric quantities
is retained in $D_0$, $\eta^2$, $\gamma_V$,
$c_1$, and $c_2$ through
\eqref{eq:confidence_radius}.

\medskip\noindent
\textbf{(b) Adiabatic slow-variable error.}
Let $y_{\mathrm{ad}}$ solve
\[
dy_{\mathrm{ad}}
=
\varepsilon g(X^*(y_{\mathrm{ad}}),y_{\mathrm{ad}})\,dt
+\sigma h(y_{\mathrm{ad}})\,dW_t,
\qquad
y_{\mathrm{ad}}(0)=y_0,
\]
driven by the same Brownian motion as the full system.
Outside $K$, the equation is interpreted using
the fixed extension of $X^*$ described in
Section~\ref{sec:slow}.

For the fixed $p>0$ in part~(a), fix
$C_{\mathrm{fast}}>0$.
For fixed geometric parameters, there exists
$C'>0$, independent of $\varepsilon$ and $\sigma$,
such that, defining
\[
\delta_\ast
=
C'\bigl(
\varepsilon+\sigma\sqrt{\varepsilon}
+\sigma^2M_{C_2}
\bigr),
\]
we have, for sufficiently small $\varepsilon$,
\begin{equation}\label{eq:adiab_err}
\begin{aligned}
&\PP\!\left(
\sup_{0\le t\le T/\varepsilon}
|y(t)-y_{\mathrm{ad}}(t)|
\le\delta_\ast,
\quad \tau_K>T/\varepsilon
\right)\\
&\qquad\ge
1-2\varepsilon^p-e^{-C_K}
-p_{\mathrm{fast}}-p_{\mathrm{mart}}.
\end{aligned}
\end{equation}
The constant $C'$ may depend on the fixed system
data, the initial-data bound, the chosen extension,
$T$, $p$, the admissible prefactor $C_0$,
and $C_{\mathrm{fast}}$.

Here $p_{\mathrm{fast}}$ is defined in
Lemma~\ref{lem:fast}, and $p_{\mathrm{mart}}$
is given by \eqref{eq:pmart_corrected},
with the degenerate cases specified in those lemmas.
For $M_\chi>0$, increasing $C_{\mathrm{fast}}$
makes $p_{\mathrm{fast}}$ arbitrarily small.

In the present martingale estimate,
$p_{\mathrm{mart}}$ is controlled through
the admissible noise prefactor $C_0$.
Indeed, if $\sigma L_h>0$, then
$\sigma\le C_0\sqrt{\varepsilon}$ implies
\[
p_{\mathrm{mart}}
\le
2m\exp\!\left(
-\frac{1}
{32mC_0^2L_h^2T e^{2L_FT}}
\right).
\]
\end{theorem}

\begin{remark}[Uniform confinement bound]\label{rem:exit_prob}
The term $e^{-C_K}$ in~\eqref{eq:path_conc} is a uniform
bound over the admissible noise regime, not an asymptotic
formula for the exit probability.
For fixed $T$ and fixed confinement data, the more explicit
estimate~\eqref{eq:exit_finite_horizon} tends to zero when
$\sigma=o(\sqrt{\varepsilon})$.
If $\sigma/\sqrt{\varepsilon}$ approaches a positive constant,
the effective noise amplitude on the slow timescale does
not vanish, so confinement need not become certain as
$\varepsilon\to0$.
The uniform formulation fixes $C_K$ by choosing $C_0$
sufficiently small, with its permitted dependence on $T$
and the confinement data.
\end{remark}

\begin{remark}[Sufficiency of the threshold]\label{rem:sufficiency}
The scale $\sigma_c(\varepsilon)$ in Theorem~\ref{thm:main}
is an explicit sufficient upper bound on the noise amplitude
for the stated concentration and error estimates.
\end{remark}

\begin{remark}[When curvature governs]
The two branches in $\sigma_c$ compare as
\[
\varepsilon^{1/4}\frac{L_{\mathrm{geom}}}{\sqrt{\lambda_0}}
< \sqrt{\varepsilon}
\;\;\Longleftrightarrow\;\;
M_{C_2}\,M_h^2 > \varepsilon^{-1/2}.
\]
Hence the curvature condition is binding precisely when the
manifold curvature is parametrically large
($M_{C_2} \gtrsim \varepsilon^{-1/2}$). For generic systems
with $M_{C_2} = O(1)$, one has $\sigma_c \sim \sqrt{\varepsilon}$,
recovering the classical scaling from
\citet{BerglundGentz2003}.
\end{remark}

\begin{remark}[Role of curvature in adiabatic error]
The term $\sigma^2 M_{C_2}$ in \eqref{eq:adiab_err}
comes from the curvature-induced systematic bias in the slow
dynamics. Under $\sigma \le \sigma_c \le C_0\sqrt{\varepsilon}$,
we have $\sigma^2 \le C_0^2 \varepsilon$, so this term is
$O(\varepsilon M_{C_2})$, which is of the same order as the
classical deterministic adiabatic error $O(\varepsilon)$.
For parameter-dependent families, the error budget is evaluated
with the corresponding geometric constants.
\end{remark}

\section{Proofs}
\label{sec:proofs}

We divide the proof into four parts: exit-time control for the
slow variable (Section~\ref{sec:exit}), the deviation dynamics
and Lyapunov function (Section~\ref{sec:deviation}), the
exponential Lyapunov bound and path concentration
(Section~\ref{sec:martingale}), and the slow-variable error
(Section~\ref{sec:slow}).

Define the fast-fibre deviation $z(t) = x(t) - X^*(y(t))$.

Define the combined stopping time
\begin{equation}\label{eq:tau}
\tau = \tau_R \wedge \tau_K, \qquad
\tau_R = \inf\{t > 0 : |z(t)| \ge R\}, \quad
\tau_K = \inf\{t > 0 : y(t) \notin K\}.
\end{equation}
All subsequent estimates hold on $[0,\tau]$, where the Lyapunov
constants are well-defined and the linearisation is valid.

\subsection{Exit-time estimate for the slow variable}
\label{sec:exit}

We first control slow-variable exit up to the fast-tube
stopping time. This estimate supplies the confinement term
in the joint probability bound.

\begin{lemma}[Slow-variable confinement]\label{lem:exit_K}
Under Assumptions~\ref{ass:NH}, \ref{ass:globs} and
\ref{ass:confinement}, for $y_0 \in K_0$ and
$\sigma \le \sigma_c(\varepsilon)$, there exists a constant
$C_K > 0$ (depending on $c_\psi, M_\psi, M_{D^2\psi}, M_h$,
$R_K - R_\star$ and independent of $\varepsilon, \sigma, T$) such
that, for $\varepsilon$ sufficiently small,
\begin{equation}\label{eq:exit_bound}
\PP\!\bigl(\tau_K \le T/\varepsilon \;\text{and}\;
\tau_K \le \tau_R\bigr)
\;\le\; \exp(-C_K).
\end{equation}
\end{lemma}

\begin{remark}[Interpretation of the bound]
Equation~\eqref{eq:exit_finite_horizon} gives an explicit
finite-horizon bound in terms of $\varepsilon$, $\sigma$,
and the confinement data.
The simpler bound $e^{-C_K}$ follows by fixing $C_K>0$
and choosing the admissible prefactor $C_0$ sufficiently small.
The independence of $C_K$ from $T$ uses the permitted
dependence of $C_0$ on $T$; it does not give a bound
uniform over arbitrarily long horizons at a fixed
noise threshold.
This confinement estimate will be combined with the
fast-deviation exit estimate to control the combined
stopping time $\tau$.
\end{remark}

\begin{proof}[Proof of Lemma~\ref{lem:exit_K}]
Work on the stopped domain $[0,\tau]$ with $\tau = \tau_R \wedge
\tau_K$ from \eqref{eq:tau}. For $s \le \tau$, we have $|z(s)| \le
R$ and $y(s) \in K$, so $\psi(y(s)) \le R_K$.
Apply It\^o's formula to $\psi(y(t))$:
\begin{align}
d\psi(y) &= \Bigl[\varepsilon\, \nabla\psi(y) \cdot g(X^*(y) + z, y)
+ \tfrac{\sigma^2}{2}\tr\bigl(h(y)h(y)^T D^2\psi(y)\bigr)\Bigr]dt
+ \sigma\,\nabla\psi(y)^T h(y)\,dW_t.\label{eq:dpsi}
\end{align}
For $y(s) \in K \setminus \mathrm{int}(K_0)$, the
confinement condition (Assumption~\ref{ass:confinement}(ii))
gives $\nabla\psi \cdot g(X^* + z, y) \le -c_\psi$ whenever
$|z| \le R_0$. Choosing the stopping radius $R \le R_0$, the
drift of $\psi$ on $\{t \le \tau\} \cap \{y(t) \notin \mathrm{int}(K_0)\}$
satisfies
\[
\varepsilon\,\nabla\psi \cdot g + \tfrac{\sigma^2}{2}
\tr(hh^T D^2\psi)
\le -\varepsilon c_\psi + \tfrac{\sigma^2}{2} M_h^2 M_{D^2\psi}.
\]

Since $\sigma \le \sigma_c \le C_0\sqrt{\varepsilon}$,
\[
\frac{\sigma^2}{2}M_h^2M_{D^2\psi}
\le
\frac{C_0^2}{2}\varepsilon M_h^2M_{D^2\psi}.
\]
Choosing the constant $C_0$ sufficiently small so that
\[
C_0^2 M_h^2 M_{D^2\psi}\le c_\psi,
\]
we obtain
\[
\varepsilon\,\nabla\psi\cdot g
+\frac{\sigma^2}{2}\operatorname{tr}(hh^TD^2\psi)
\le -\frac{\varepsilon c_\psi}{2}.
\]

Let $L=T/\varepsilon$ and
$\delta_\psi=R_K-R_\star>0$.
The stopped martingale part of~\eqref{eq:dpsi} is
\[
M(t)=\sigma\int_0^{t\wedge\tau}
\nabla\psi(y(s))^T h(y(s))\,dW_s.
\]
Its quadratic variation satisfies
\[
\langle M\rangle_L
\le \frac{\sigma^2 M_\psi^2 M_h^2 T}{\varepsilon}.
\]

Consider the event
\[
\mathcal E=\{\tau_K\le L,\ \tau_K\le\tau_R\}.
\]
On $\mathcal E$, continuity gives
$\psi(y(\tau_K))=R_K$ and a last crossing time
\[
\tau_0=\sup\{s\le\tau_K:\psi(y(s))=R_\star\}.
\]
We use $\tau_0$ only in pathwise identities;
no stopping-time property is required.
On $(\tau_0,\tau_K)$, the process lies outside $K_0$,
where the drift in~\eqref{eq:dpsi} is non-positive.
Consequently,
\[
\delta_\psi
\le M(\tau_K)-M(\tau_0)
\le 2\sup_{0\le t\le L}|M(t)|.
\]
The exponential martingale inequality therefore yields
\begin{equation}\label{eq:exit_finite_horizon}
\PP(\mathcal E)
\le
2\exp\!\left(
-\frac{\varepsilon\delta_\psi^2}
{8T\sigma^2 M_\psi^2 M_h^2}
\right).
\end{equation}

Fix a positive constant $C_K$ depending only on the
confinement data, for instance $C_K=1$.
In addition to the drift condition above, choose $C_0$
sufficiently small that
\[
C_0^2
\le
\frac{\delta_\psi^2}
{8T M_\psi^2 M_h^2(C_K+\log 2)}.
\]
This is consistent with the allowed dependence of $C_0$
on $T$ in~\eqref{eq:sigmac}.
Since $\sigma\le C_0\sqrt{\varepsilon}$,
equation~\eqref{eq:exit_finite_horizon} gives
$\PP(\mathcal E)\le e^{-C_K}$, proving~\eqref{eq:exit_bound}.
The time-horizon dependence is absorbed into the admissible
prefactor $C_0$, while $C_K$ is fixed independently of
$\varepsilon$, $\sigma$, and $T$.
\end{proof}

\begin{remark}[Large deviation perspective]
The factor $\varepsilon/\sigma^2$ in
\eqref{eq:exit_finite_horizon} is explained by the slow-time
rescaling $s=\varepsilon t$.
Set $Y_s=y(s/\varepsilon)$, $X_s=x(s/\varepsilon)$,
and $B_s=\sqrt{\varepsilon}\,W_{s/\varepsilon}$.
Then $B_s$ is a standard Brownian motion and
\[
dY_s=g(X_s,Y_s)\,ds
+\frac{\sigma}{\sqrt{\varepsilon}}h(Y_s)\,dB_s.
\]
Thus the effective noise variance on the slow timescale
is $\sigma^2/\varepsilon$.

For fixed $T$ and fixed confinement data,
the bound in~\eqref{eq:exit_finite_horizon} tends to zero
when $\sigma=o(\sqrt{\varepsilon})$.
The present argument gives a finite-time martingale
upper bound, rather than identifying a
Freidlin--Wentzell quasipotential.

For background on quasipotentials and exit problems, see
\citet{FreidlinWentzell2012}.
Related large-deviation results for stochastic averaging
and multiscale diffusions can be found in
\citet{Veretennikov1999,DupuisSpiliopoulos2012}.
\end{remark}

\paragraph{Short-time exit without confinement.}
Without Assumption~\ref{ass:confinement}, a shorter-time
estimate follows from the boundedness of $g$ and $h$.
Let $y_0\in\mathrm{int}(K)$ and
$d_K=\operatorname{dist}(y_0,\partial K)>0$.
Choose $R$ sufficiently small that the stopped tube lies
in $U$ from Definition~\ref{def:constants}.
For any deterministic observation time $t_{\mathrm{obs}}>0$
satisfying
\[
\varepsilon M_g t_{\mathrm{obs}}\le d_K/2,
\]
we have
\begin{equation}\label{eq:exit_short_time}
\PP\!\left(
\tau_K\le t_{\mathrm{obs}},\ \tau_K\le\tau_R
\right)
\le
2m\exp\!\left(
-\frac{d_K^2}
{8m\sigma^2 M_h^2t_{\mathrm{obs}}}
\right).
\end{equation}

Indeed, define the stopped vector martingale
\[
M_y(t)=\sigma\int_0^{t\wedge\tau}h(y(s))\,dW_s.
\]
On the exit event in~\eqref{eq:exit_short_time},
integration of the slow equation gives
\[
d_K
\le |y(\tau_K)-y_0|
\le \varepsilon M_g t_{\mathrm{obs}}+|M_y(\tau_K)|.
\]
Hence $\sup_{t\le t_{\mathrm{obs}}}|M_y(t)|\ge d_K/2$.
Each component of $M_y$ has quadratic variation bounded
by $\sigma^2M_h^2t_{\mathrm{obs}}$.
Applying the exponential martingale inequality to each
component and taking a union bound proves
\eqref{eq:exit_short_time}.

For fixed coefficient bounds and a fixed positive
boundary distance, this estimate tends to zero whenever
$\varepsilon t_{\mathrm{obs}}\to0$ and
$\sigma^2t_{\mathrm{obs}}\to0$.

\subsection{Deviation dynamics and Lyapunov function}
\label{sec:deviation}

\begin{remark}[Choice of $M_0$ over $M_\varepsilon$]
We formulate the estimate around the critical manifold $M_0$
because $X^*$ and its derivatives can be obtained directly
from the equilibrium equation $f(X^*(y),y)=0$.

For comparison, suppose a deterministic Fenichel manifold
is locally represented by a sufficiently smooth graph
$x=X^\varepsilon(y)$. Its invariance equation is
\[
f(X^\varepsilon(y),y)
=
\varepsilon DX^\varepsilon(y)
g(X^\varepsilon(y),y).
\]
For the deviation $\zeta=x-X^\varepsilon(y)$, this identity
cancels the deterministic forcing at $\zeta=0$.
At $\zeta=0$, It\^o's formula leaves the drift contribution
\[
-\frac{\sigma^2}{2}D^2X^\varepsilon(y)
[h(y)\otimes h(y)]
\]
and the diffusion coefficient $-\sigma DX^\varepsilon(y)h(y)$.

Thus, using $M_\varepsilon$ removes the deterministic
$O(\varepsilon)$ forcing but retains the noise-induced terms.
This can improve estimates when $\sigma\ll\varepsilon$,
provided the initial deviation from $M_\varepsilon$
is controlled accordingly.
The $O(\varepsilon)$ contribution retained in
Theorem~\ref{thm:main}(a) accounts for deterministic tracking
relative to the explicitly computable critical manifold $M_0$.
\end{remark}

We first derive the exact evolution equation for the deviation
$z(t)$. This makes explicit both the noise transmitted through
$DX^*$ and the It\^o correction involving $D^2X^*$, providing
the starting point for the Lyapunov estimates below.

\begin{lemma}[Exact deviation SDE]\label{lem:exact_z}
The process $z$ satisfies
\begin{equation}\label{eq:z_exact}
dz = \Bigl[f(X^*\!+\!z,y) - \varepsilon DX^* g(X^*\!+\!z,y)
- \tfrac{\sigma^2}{2}D^2X^*[h\otimes h]\Bigr]dt
- \sigma DX^* h(y)\,dW_t,
\end{equation}
where all coefficients are evaluated at $y = y(t)$.
The diffusion $\Sigma_z(y) = -\sigma DX^*(y) h(y)$
is independent of $z$.
\end{lemma}

\begin{proof}
Apply It\^o's formula to $X^*(y(t))$. Since
$dy = \varepsilon g(x,y)\,dt + \sigma h(y)\,dW_t$, the
It\^o correction is
$\tfrac{\sigma^2}{2}\sum_{j,k}\partial^2_{y_j y_k}X^*
(hh^T)_{jk}\,dt = \tfrac{\sigma^2}{2}D^2X^*[h\otimes h]\,dt$.
Subtracting $d(X^*(y))$ from $dx = f(x,y)\,dt$ gives
\eqref{eq:z_exact}.
\end{proof}

We next isolate the stable linear part $A(y)z$ and the
forcing terms of orders $\varepsilon$ and $\sigma^2$.
The remaining terms are bounded in powers of $|z|$
for use in the Lyapunov estimate.

\begin{lemma}[Taylor expansion of the $z$-drift]\label{lem:expansion}
Define $b_\varepsilon(y) = -\varepsilon DX^*(y) g_0(y)$
and $b_\sigma(y) = -\frac{\sigma^2}{2}D^2X^*(y)[h(y)\otimes h(y)]$
where $g_0(y) = g(X^*(y),y)$. Then \eqref{eq:z_exact} rewrites as
\begin{equation}
dz = \bigl[A(y)z + b(y) + N(z,y)\bigr]dt + \Sigma_z(y)\,dW_t,
\end{equation}
where $b(y) = b_\varepsilon(y) + b_\sigma(y)$. Note that the remainder term 
$N(z,y)$ contains the $\varepsilon$-dependent linear correction 
$-\varepsilon DX^*(y) D_xg(X^*(y),y)\,z$ in addition to genuine nonlinear 
Taylor remainders, and satisfies
\[
\|N(z,y)\| \le \Gamma_1|z| + \Gamma_2|z|^2 + C_f|z|^3
\]
with $\Gamma_1 = \varepsilon M_C M_{D_xg}$ and
$\Gamma_2 = \frac{M_B}{2} + \varepsilon M_C C_g$, and
\[
\|b(y)\| \le \beta := \varepsilon M_C M_g
+ \tfrac{\sigma^2}{2}M_{C_2}M_h^2.
\]
\end{lemma}

\begin{proof}
Expand $f(X^*+z,y) = A(y)z + \tfrac12 B(y)[z,z] + R_f(z,y)$
and $g(X^*+z,y) = g_0 + D_xg \cdot z + R_g(z,y)$. Substituting
into \eqref{eq:z_exact} and grouping terms gives the stated
form with
\[
N(z,y) = \tfrac12 B(y)[z,z] - \varepsilon DX^* D_xg z
- \varepsilon DX^* R_g + R_f .
\]
Applying the triangle inequality and the bounds in
Definition~\ref{def:constants} yields
\[
\begin{aligned}
\|N(z,y)\|
&\le \tfrac12 \|B(y)[z,z]\| + \varepsilon\|DX^*\|\,\|D_xg\|\,\|z\|
   + \varepsilon\|DX^*\|\,\|R_g\| + \|R_f\| \\
&\le \tfrac{M_B}{2}|z|^2 + \varepsilon M_C M_{D_xg}|z|
   + \varepsilon M_C C_g|z|^2 + C_f|z|^3,
\end{aligned}
\]
which is exactly $\|N(z,y)\| \le \Gamma_1|z| + \Gamma_2|z|^2 + C_f|z|^3$
with the constants $\Gamma_1,\Gamma_2$ stated above.
\end{proof}

We measure the deviation using $V(z,y)=z^T P(y)z$,
for which the Lyapunov equation~\eqref{eq:lyap} yields
the dissipative contribution $-|z|^2$.
Since $z$ and $y$ share the same Brownian driver,
the It\^o expansion must also retain their mixed
quadratic-variation terms.

\begin{lemma}[It\^o expansion of $V = z^T P(y) z$]
\label{lem:ito_V}
Write $h_\ell(y)$ for the $\ell$th column of $h(y)$,
where $\ell=1,\dots,d$.
The Lyapunov function $V(z,y)=z^TP(y)z$ satisfies
\[
dV=\mathcal{L}V\,dt+dM_V,
\]
where the drift decomposes as
\begin{align}
\mathcal{L}V
&=-|z|^2+2z^TP(y)b(y)+2z^TP(y)N(z,y)
\notag\\
&\quad+\varepsilon z^T
DP(y)[g(X^*(y)+z,y)]z
\notag\\
&\quad+\sigma^2\sum_{\ell=1}^d
\bigl(DX^*(y)h_\ell(y)\bigr)^T
P(y)\bigl(DX^*(y)h_\ell(y)\bigr)
\notag\\
&\quad+\frac{\sigma^2}{2}\sum_{\ell=1}^d
z^TD^2P(y)[h_\ell(y),h_\ell(y)]z
\notag\\
&\quad-2\sigma^2\sum_{\ell=1}^d
z^TDP(y)[h_\ell(y)]DX^*(y)h_\ell(y).
\label{eq:LV}
\end{align}
The martingale part is
\begin{equation}\label{eq:MV}
\begin{aligned}
dM_V
=\sigma\sum_{\ell=1}^d
\Bigl[
&-2z^TP(y)DX^*(y)h_\ell(y)\\
&+z^TDP(y)[h_\ell(y)]z
\Bigr]\,dW_t^\ell.
\end{aligned}
\end{equation}
\end{lemma}

\begin{proof}
Apply It\^o's formula to $V(z,y)$ using the joint process
$(z(t),y(t))$. The terms arise as: (i) $2(Pz)^T(Az) = z^T(A^T P + PA)z
= -|z|^2$; (ii) $2(Pz)^T b$; (iii) $2(Pz)^T N$; (iv)
the $DP$-It\^o correction from $dy$;
(v) half the trace from $\Sigma_z \Sigma_z^T$;
(vi) half the second-order It\^o correction from $d\langle y \rangle$;
(vii) the cross-It\^o correction from
$d\langle z, y\rangle = \Sigma_z \Sigma_y^T\,dt
= -\sigma^2 DX^* hh^T\,dt$.
\end{proof}

We now convert the preceding It\^o expansion into a scalar
drift inequality for $V$. Under the smallness conditions
stated below, the terms depending on $z$ are absorbed
into the dissipative contribution, while the remaining
forcing terms are collected in the constant $D_0$.

\begin{lemma}[Foster--Lyapunov bound]\label{lem:foster}
Suppose $\tilde\lambda>0$. On the domain
$\mathcal{D}_R = \{|z| \le R, y \in K\}$, choose $R$ so that
\begin{equation}\label{eq:Rcond}
2c_2\Bigl(\Gamma_2 + \varepsilon M_{DP} M_{D_xg}\Bigr)R
+ 2c_2 C_f R^2 \;\le\; \frac{\tilde\lambda}{4},
\end{equation}
the effective decay rate
\[
\tilde\lambda := 1 - 2c_2\Gamma_1 - \varepsilon M_{DP}M_g
- \frac{\sigma^2}{2}M_{D^2P}M_h^2
\]
is positive by hypothesis, and $\mathcal{L}V$ satisfies
\begin{equation}\label{eq:foster}
\mathcal{L}V \;\le\; -\gamma_V V + D_0,
\qquad \gamma_V = \frac{\tilde\lambda}{2c_2},
\end{equation}
with forcing constant
\begin{equation}\label{eq:D0}
D_0 = \sigma^2 c_2 M_C^2 M_h^2
+ \frac{4}{\tilde\lambda}\bigl(c_2\beta + \sigma^2 M_{DP} M_C M_h^2\bigr)^2.
\end{equation}
\end{lemma}

\begin{remark}[Validity of $\tilde\lambda > 0$]\label{rem:lambda_pos}
Since $K_A$, $\|DA\|_K$, and $\|D^2A\|_K$ are fixed system constants,
$M_{DP}$ and $M_{D^2P}$ are finite. The terms $\varepsilon M_{DP} M_g$
and $\sigma^2 M_{D^2P} M_h^2/2$ in $\tilde\lambda$ vanish as
$\varepsilon, \sigma \to 0$. Hence $\tilde\lambda > 0$ holds for all
$\varepsilon \le \varepsilon_0$ and $\sigma \le \sigma_0$ with
$\varepsilon_0, \sigma_0$ depending only on these system constants.
\end{remark}

\begin{proof}
Bound each term in \eqref{eq:LV} using the constants of
Definition~\ref{def:constants} and Appendix~\ref{sec:appendix}:
\begin{align*}
\text{(ii)} &\le 2c_2\beta |z|, \\
\text{(iii)} &\le 2c_2\bigl(\Gamma_1|z|^2 + \Gamma_2|z|^3 + C_f|z|^4\bigr), \\
\text{(iv)} &\le \varepsilon M_{DP}M_g|z|^2, \\
\text{(v)} &\le \sigma^2 c_2 M_C^2 M_h^2 =: \Theta_0, \\
\text{(vi)} &\le \tfrac{\sigma^2}{2}M_{D^2P}M_h^2|z|^2, \\
\text{(vii)} &\le 2\sigma^2 M_{DP} M_C M_h^2 |z|.
\end{align*}

We group the quadratic terms explicitly to isolate the primary decay rate 
$-\tilde\lambda|z|^2$. Combine the two linear contributions in $|z|$ to obtain 
$2\bigl(c_2\beta + \sigma^2 M_{DP} M_C M_h^2\bigr)|z|$. Applying Young's inequality 
$2uv \le \frac{\tilde\lambda}{4}u^2 + \frac{4}{\tilde\lambda}v^2$ with $u = |z|$ 
and $v = c_2\beta + \sigma^2 M_{DP} M_C M_h^2$ yields exactly the term 
$\frac{4}{\tilde\lambda}\bigl(c_2\beta + \sigma^2 M_{DP} M_C M_h^2\bigr)^2$ 
which is added to $\Theta_0$ to form $D_0$, while generating a positive 
quadratic contribution $\frac{\tilde\lambda}{4}|z|^2$.

For $|z|\le R$, the stopping-radius condition
\eqref{eq:Rcond} implies
\[
\begin{aligned}
2c_2\Gamma_2|z|^3+2c_2C_f|z|^4
&\le
\bigl(2c_2\Gamma_2R+2c_2C_fR^2\bigr)|z|^2\\
&\le \frac{\tilde\lambda}{4}|z|^2.
\end{aligned}
\]

Collecting all quadratic contributions therefore gives
\[
-\tilde\lambda|z|^2
+\frac{\tilde\lambda}{4}|z|^2
+\frac{\tilde\lambda}{4}|z|^2
=
-\frac{\tilde\lambda}{2}|z|^2,
\]
where the two positive terms arise respectively from Young's
inequality and from absorbing the cubic and quartic remainders.

Since
\[
V(z,y)=z^TP(y)z\le c_2|z|^2,
\]
we have
\[
|z|^2\ge \frac{V}{c_2}.
\]
Hence
\[
-\frac{\tilde\lambda}{2}|z|^2
\le
-\frac{\tilde\lambda}{2c_2}V
=
-\gamma_VV,
\]
which proves \eqref{eq:foster}.
\end{proof}

\subsection{Exponential martingale and path concentration}
\label{sec:martingale}

The drift estimate alone does not control the probability
of large fluctuations. We therefore first bound the
quadratic variation of $M_V$ in terms of $V$, and then
combine the two estimates in a stopped exponential
Lyapunov argument.

\begin{lemma}[Quadratic variation bound]\label{lem:QV}
On $\mathcal{D}_R$, the predictable quadratic variation of
$M_V$ satisfies $d\langle M_V\rangle / dt \le \eta^2 V$ with
\[
\eta^2 = \frac{2\sigma^2}{c_1}\Bigl(4c_2^2 M_C^2 M_h^2
+ M_{DP}^2 M_h^2 R^2\Bigr).
\]
\end{lemma}

\begin{proof}
From~\eqref{eq:MV}, using $(a+b)^2 \le 2a^2+2b^2$
and $|z| \le R$ on $\mathcal{D}_R$, we obtain
\[
\begin{aligned}
\frac{d\langle M_V\rangle}{dt}
&= \sigma^2\sum_{\ell=1}^d
\bigl[-2z^TP(y)DX^*h_\ell+z^TDP[h_\ell]z\bigr]^2 \\
&\le 2\sigma^2
\Bigl(4c_2^2M_C^2M_h^2 + M_{DP}^2M_h^2R^2\Bigr)|z|^2.
\end{aligned}
\]
Using $V \ge c_1|z|^2$ gives
\[
\frac{d\langle M_V\rangle}{dt}
\le \frac{2\sigma^2}{c_1}
\Bigl(4c_2^2M_C^2M_h^2 + M_{DP}^2M_h^2R^2\Bigr)V,
\]
which yields the stated value of $\eta^2$.
\end{proof}

We retain the dissipative drift when estimating the exponential
transform, and then stop at the first crossing of an energy level.

\begin{lemma}[Dissipative exponential Lyapunov estimate]
\label{lem:superm}
Suppose $\gamma_V>0$ and $\eta^2>0$, and set
$\alpha=\gamma_V/\eta^2$.
On $\mathcal D_R$, the function
$E_\alpha(z,y)=e^{\alpha V(z,y)}$ satisfies
\begin{equation}\label{eq:exp_generator_bound}
\begin{aligned}
\mathcal L E_\alpha
&\le \alpha e^{\alpha V}
\left(D_0-\frac{\gamma_V}{2}V\right)\\
&\le \frac{\gamma_V}{2e}
\exp\!\left(\frac{2D_0}{\eta^2}\right).
\end{aligned}
\end{equation}
\end{lemma}

\begin{proof}
By It\^o's formula and
Lemmas~\ref{lem:foster} and~\ref{lem:QV},
\[
\begin{aligned}
\mathcal L E_\alpha
&=\alpha e^{\alpha V}
\left(
\mathcal LV+
\frac{\alpha}{2}\frac{d\langle M_V\rangle}{dt}
\right)\\
&\le \alpha e^{\alpha V}
\left(
-\gamma_V V+D_0+\frac{\alpha\eta^2}{2}V
\right)\\
&=\alpha e^{\alpha V}
\left(D_0-\frac{\gamma_V}{2}V\right).
\end{aligned}
\]
Set $u=\alpha(V-2D_0/\gamma_V)$.
The last expression equals
\[
\frac{\gamma_V}{2}
\exp\!\left(\frac{2D_0}{\eta^2}\right)(-u)e^u.
\]
Since $(-u)e^u\le e^{-1}$ for every $u\in\mathbb R$,
the claimed bound follows.
\end{proof}

We next stop at the first crossing of a prescribed energy
level and use the exponential generator estimate to obtain
a maximal tail bound for an arbitrary observation time $t$.

\begin{lemma}[Maximal tail bound for $V$]
\label{lem:tail_V}
Suppose $\gamma_V>0$ and $\eta^2>0$, and set
$\alpha=\gamma_V/\eta^2$.
For every $v>0$ and $t>0$,
\begin{align}
\PP\!\left(
\sup_{0\le s\le t\wedge\tau}V(s)\ge v
\right)
&\le e^{-\alpha v}
\left[
e^{\alpha V(0)}
+\frac{\gamma_Vt}{2e}e^{2D_0/\eta^2}
\right]
\label{eq:tail_V_exact}\\
&\le (1+\gamma_Vt)
\exp\!\left[
-\frac{\gamma_V}{\eta^2}
\left(v-V(0)-\frac{2D_0}{\gamma_V}\right)
\right].
\label{eq:tail_V}
\end{align}
\end{lemma}

\begin{proof}
If $v\le V(0)$, both upper bounds are at least one,
so the claims are immediate.
Suppose $v>V(0)$ and define
\[
\tau_v=\inf\{s\ge0:V(s\wedge\tau)\ge v\},
\qquad
\theta=t\wedge\tau\wedge\tau_v.
\]
By It\^o's formula and
\eqref{eq:exp_generator_bound},
\[
\EE e^{\alpha V(\theta)}
\le e^{\alpha V(0)}
+\frac{\gamma_Vt}{2e}
\exp\!\left(\frac{2D_0}{\eta^2}\right).
\]
The stopped martingale term has zero expectation:
the coefficients and $V$ are bounded on the
localized domain, and $\theta\le t$.

By continuity,
\[
\left\{
\sup_{0\le s\le t\wedge\tau}V(s)\ge v
\right\}
=
\{\tau_v\le t\wedge\tau\}.
\]
On this event, $V(\theta)=v$.
Consequently,
\[
e^{\alpha v}
\PP(\tau_v\le t\wedge\tau)
\le \EE e^{\alpha V(\theta)},
\]
which proves \eqref{eq:tail_V_exact}.

Finally, since $V(0)\ge0$ and $D_0\ge0$,
\[
e^{\alpha V(0)}
+\frac{\gamma_Vt}{2e}e^{2D_0/\eta^2}
\le
(1+\gamma_Vt)
\exp\!\left[
\alpha\left(V(0)+\frac{2D_0}{\gamma_V}\right)
\right].
\]
Substitution gives \eqref{eq:tail_V}.
\end{proof}

We now translate the maximal bound for $V$ into a bound
for $|z|$, using the comparison between the Lyapunov
function and the squared Euclidean norm.
This expresses the estimate directly in terms of
the deviation from the critical manifold.

\begin{proposition}[Path concentration in $|z|$]
\label{prop:path_conc}
Suppose $\gamma_V>0$ and $\eta^2>0$.
For every $r>0$ and $t>0$,
\begin{equation}\label{eq:pathz}
\begin{aligned}
&\PP\!\left(
\sup_{0\le s\le t\wedge\tau}|z(s)|\ge r
\right)\\
&\quad\le (1+\gamma_Vt)
\exp\!\left[
-\frac{\gamma_V}{\eta^2}
\left(
c_1r^2-c_2|z_0|^2-\frac{2D_0}{\gamma_V}
\right)
\right].
\end{aligned}
\end{equation}
\end{proposition}

\begin{proof}
The comparison bounds \eqref{eq:Pbounds} give
\[
V(s)\ge c_1|z(s)|^2,
\qquad
V(0)\le c_2|z_0|^2.
\]
Consequently,
\[
\left\{
\sup_{0\le s\le t\wedge\tau}|z(s)|\ge r
\right\}
\subseteq
\left\{
\sup_{0\le s\le t\wedge\tau}V(s)\ge c_1r^2
\right\}.
\]
Apply \eqref{eq:tail_V} with $v=c_1r^2$
and use $V(0)\le c_2|z_0|^2$ to obtain
\eqref{eq:pathz}.
\end{proof}

We now choose the radius for a prescribed failure probability.
The localization condition is imposed on this radius itself.

\begin{lemma}[Concentration with an explicit confidence level]
\label{lem:concentration}
Suppose $\gamma_V>0$.
Let $L_\varepsilon=T/\varepsilon$ and $\rho\in(0,1)$.
Define
\begin{equation}\label{eq:confidence_radius}
\begin{aligned}
r_\rho(\varepsilon,\sigma)^2
=\frac{1}{c_1}\Biggl[
&c_2|z_0|^2+\frac{2D_0}{\gamma_V}\\
&+\frac{\eta^2}{\gamma_V}
\log\!\left(
\frac{1+\gamma_VL_\varepsilon}{\rho}
\right)
\Biggr].
\end{aligned}
\end{equation}
If $r_\rho<R$, then
\begin{equation}\label{eq:confidence_stopped}
\PP\!\left(
\sup_{0\le s\le L_\varepsilon\wedge\tau}
|z(s)|>r_\rho
\right)\le\rho.
\end{equation}
If the confinement estimate~\eqref{eq:exit_bound}
also holds, then
\begin{equation}\label{eq:confidence_joint}
\PP\!\left(
\sup_{0\le s\le L_\varepsilon}|z(s)|\le r_\rho,
\quad \tau_K>L_\varepsilon
\right)
\ge 1-\rho-e^{-C_K}.
\end{equation}
For $\eta=0$, the same conclusions hold with
the logarithmic term in~\eqref{eq:confidence_radius}
set to zero.
\end{lemma}

\begin{proof}
For $\eta^2>0$, substitute
\eqref{eq:confidence_radius} into
\eqref{eq:pathz} with $t=L_\varepsilon$.
This gives \eqref{eq:confidence_stopped}.

If $\eta=0$, the quadratic variation of $M_V$
vanishes up to $\tau$, so the martingale increment
is zero. The Foster--Lyapunov inequality gives
\[
\begin{aligned}
V(s)
&\le e^{-\gamma_Vs}V(0)
+\frac{D_0}{\gamma_V}(1-e^{-\gamma_Vs})\\
&\le \max\!\left\{
V(0),\frac{D_0}{\gamma_V}
\right\},
\qquad 0\le s\le\tau.
\end{aligned}
\]
Together with \eqref{eq:Pbounds}, this proves
\eqref{eq:confidence_stopped} also in this case.

Since $r_\rho<R$, an exit through $|z|=R$
before $\tau_K$ and by time $L_\varepsilon$
implies
\[
\sup_{0\le s\le L_\varepsilon\wedge\tau}
|z(s)|>r_\rho.
\]
Consequently,
\[
\begin{aligned}
&\left\{
\sup_{0\le s\le L_\varepsilon}|z(s)|>r_\rho
\right\}
\cup\{\tau_K\le L_\varepsilon\}\\
&\quad\subseteq
\left\{
\sup_{0\le s\le L_\varepsilon\wedge\tau}
|z(s)|>r_\rho
\right\}\\
&\qquad\quad\cup
\{\tau_K\le L_\varepsilon,\ \tau_K\le\tau_R\}.
\end{aligned}
\]
The union bound and the confinement estimate
\eqref{eq:exit_bound} yield
\eqref{eq:confidence_joint}.
\end{proof}

\begin{remark}[Noise scale, observation time, and confidence]
\label{rem:confidence_tradeoff}
The maximal estimate \eqref{eq:pathz} is exponential in the
ratio between the squared tube radius and the effective noise
variance $\eta^2/\gamma_V$. The observation horizon enters
through the prefactor $1+\gamma_Vt$. On the slow timescale,
$\varepsilon$ enters through $t=T/\varepsilon$ and the chosen
confidence level, while the noise dependence remains explicit
in $\eta^2$.

For fixed system parameters, $|z_0|=O(\varepsilon)$ and
$\sigma\le C_0\sqrt{\varepsilon}$ give
$D_0=O(\varepsilon^2+\sigma^2+\sigma^4)$,
$\eta^2=O(\sigma^2)$, and a uniform positive lower bound
on $\gamma_V$.
On $[0,T/\varepsilon]$, choosing $\rho=\varepsilon^p$
with fixed $p>0$ yields the logarithmic noise term
$\sigma\sqrt{\log(1/\varepsilon)}$ in $r_\rho$.
The constants in this radius depend on $p$.
Choosing instead $\rho=e^{-\kappa/\varepsilon}$ adds
$\kappa\eta^2/(\gamma_V\varepsilon)$ to its squared
energy budget. If $\eta^2/\gamma_V\asymp\sigma^2$,
this corresponds to a noise contribution of order
$\sigma/\sqrt{\varepsilon}$.
Thus exponential confidence and the smaller logarithmic tube
are different radius--confidence choices.
The slow-variable confinement bound $e^{-C_K}$ is a separate
term; the fast-variable tail bound itself remains exponential.
\end{remark}

The slow-variable estimate also requires control of the
time integral of $V$. We next derive a stopped
time-integral bound with a prescribed failure probability.

\paragraph{Role of the two-branch sufficient noise scale.}
The two restrictions in \eqref{eq:sigmac} control the diffusion
energy and the square of the curvature forcing, respectively.
To see this explicitly, put
\[
a_0=c_2M_C^2M_h^2,\qquad
q_0=M_CM_g,\qquad
b_0=M_{DP}M_CM_h^2,\qquad
q_2=M_{C_2}M_h^2.
\]
For fixed system parameters and sufficiently small $\varepsilon$,
$\tilde\lambda\ge1/2$.
The condition $\sigma\le\sigma_c(\varepsilon)$ gives
\[
\sigma^2\le C_0^2\varepsilon,
\qquad
\sigma^2q_2\le C_0^2\sqrt{\varepsilon}.
\]
The second inequality is automatic if $q_2=0$.
Since $\beta=\varepsilon q_0+\sigma^2q_2/2$, equation~\eqref{eq:D0}
and $(u+v)^2\le2u^2+2v^2$ imply
\begin{equation}\label{eq:threshold_energy_control}
\frac{D_0}{\varepsilon}
\le
 a_0C_0^2+4c_2^2C_0^4
 +16\varepsilon(c_2q_0+b_0C_0^2)^2.
\end{equation}
Thus the two-branch scale supplies a sufficient bounded-energy regime.
For parameter-dependent families the coefficients in this inequality
must be tracked explicitly, as in Example~\ref{ex:sigmoid}.

\begin{lemma}[Stopped time-integral bound on $V$]
\label{lem:timeavg}
Suppose $\gamma_V>0$.
Let $L_\varepsilon=T/\varepsilon$,
$\rho_E\in(0,1)$, and define
\[
J(t)=\int_0^{t\wedge\tau}V(s)\,ds.
\]
Set
\begin{equation}\label{eq:energy_budget}
B_E(\rho_E)
=\frac{2\bigl(V(0)+D_0L_\varepsilon\bigr)}{\gamma_V}
+\frac{2\eta^2}{\gamma_V^2}
\log\frac{1}{\rho_E}.
\end{equation}
Then
\[
\PP\!\left(
J(L_\varepsilon)>B_E(\rho_E)
\right)\le\rho_E.
\]
In particular, on
\[
\mathcal H
=\{J(L_\varepsilon)\le B_E(\rho_E)\}
\]
and $\{\tau>L_\varepsilon\}$, we have
\[
\int_0^{L_\varepsilon}V(s)\,ds
\le B_E(\rho_E).
\]
\end{lemma}

\begin{proof}
Integrating \eqref{eq:foster} up to $t\wedge\tau$
and using $V\ge0$ gives
\[
\gamma_VJ(t)
\le V(0)+D_0(t\wedge\tau)+M_V(t\wedge\tau),
\]
where the martingale is normalized by $M_V(0)=0$.

Suppose first that $\eta^2>0$.
Put
\[
K_E=V(0)+D_0L_\varepsilon,
\qquad B=B_E(\rho_E),
\]
and define the stopping time
\[
\theta
=\tau\wedge L_\varepsilon
\wedge\inf\{t\ge0:J(t)\ge B\}.
\]
By continuity and the quadratic-variation bound,
\[
J(\theta)\le B,
\qquad
\langle M_V\rangle_\theta\le\eta^2B.
\]
On $\{J(L_\varepsilon)>B\}$,
the level $B$ is reached before stopping, and hence
\[
J(\theta)=B,
\qquad
M_V(\theta)\ge\gamma_VB-K_E.
\]

For $\lambda=\gamma_V/\eta^2$, the non-negative
exponential local martingale
\[
\mathcal E_\lambda(t)
=\exp\!\left(
\lambda M_V(t\wedge\tau)
-\frac{\lambda^2}{2}
\langle M_V\rangle_{t\wedge\tau}
\right)
\]
is a supermartingale.
Since $\theta\le L_\varepsilon$,
optional stopping gives
$\EE\mathcal E_\lambda(\theta)\le1$.
Therefore,
\[
\begin{aligned}
\PP\!\left(J(L_\varepsilon)>B\right)
&\le
\exp\!\left[
-\lambda(\gamma_VB-K_E)
+\frac{\lambda^2\eta^2B}{2}
\right]\\
&=
\exp\!\left[
-\frac{\gamma_V^2B}{2\eta^2}
+\frac{\gamma_VK_E}{\eta^2}
\right]
=\rho_E.
\end{aligned}
\]

If $\eta=0$, the martingale increment vanishes.
The integrated drift inequality then gives
\[
J(L_\varepsilon)
\le\frac{K_E}{\gamma_V}
\le B_E(\rho_E).
\]
This proves the result also in that case.
\end{proof}

\subsection{Slow-variable adiabatic error}
\label{sec:slow}

Use the smooth extension of $X^*$ introduced in
Section~\ref{sec:manifold_norms}, chosen with bounded first
derivative by a cutoff outside a neighbourhood of $K$.
Set
\[
\Delta y=y-y_{\mathrm{ad}},\qquad
F(y)=g(X^*(y),y),\qquad
L_F=L_g\left(1+\sup_{\RR^m}\|DX^*\|\right).
\]
Then $F$ is globally Lipschitz with constant $L_F$.
All fast-manifold identities and bounds are used only on $K$.

We compare the full and adiabatic slow equations driven
by the same Brownian motion. The following decomposition
separates the systematic offset, the remaining fast
fluctuations, and the difference between the diffusion
coefficients.

\begin{lemma}[Stopped decomposition of $\Delta y$]
\label{lem:slow_decomposition}
Let $\vartheta\le\tau$ be a stopping time up to
which both slow solutions are defined.
Set
\[
G(y)=D_xg(X^*(y),y),
\qquad
F(y)=g(X^*(y),y),
\]
and
\[
R_g(z,y)
=
g(X^*(y)+z,y)-F(y)-G(y)z.
\]
Then
\begin{equation}\label{eq:slow_decomposition}
\begin{aligned}
\Delta y(t\wedge\vartheta)
={}&
\Delta y_{\mathrm{sys}}(t\wedge\vartheta)
+\Delta y_{\mathrm{fast}}(t\wedge\vartheta)\\
&+\Delta y_{\mathrm{mart}}(t\wedge\vartheta)\\
&+\varepsilon\int_0^{t\wedge\vartheta}
\bigl[F(y(s))-F(y_{\mathrm{ad}}(s))\bigr]\,ds\\
&+\mathrm{Rem}_{\vartheta}(t),
\end{aligned}
\end{equation}
where
\[
\begin{aligned}
\Delta y_{\mathrm{sys}}(t\wedge\vartheta)
&=
\varepsilon\int_0^{t\wedge\vartheta}
G(y(s))\mu_z(y(s))\,ds,\\
\Delta y_{\mathrm{fast}}(t\wedge\vartheta)
&=
\varepsilon\int_0^{t\wedge\vartheta}
G(y(s))\bigl(z(s)-\mu_z(y(s))\bigr)\,ds,\\
\Delta y_{\mathrm{mart}}(t\wedge\vartheta)
&=
\sigma\int_0^{t\wedge\vartheta}
\bigl[h(y(s))-h(y_{\mathrm{ad}}(s))\bigr]\,dW_s,
\end{aligned}
\]
with $\mu_z(y)=-A(y)^{-1}b(y)$, and
\[
\mathrm{Rem}_{\vartheta}(t)
=
\varepsilon\int_0^{t\wedge\vartheta}
R_g(z(s),y(s))\,ds.
\]
The remainder satisfies
\begin{equation}\label{eq:slow_taylor_remainder}
|\mathrm{Rem}_{\vartheta}(t)|
\le
\frac{\varepsilon C_g}{c_1}J(t).
\end{equation}
In particular, on the event $\mathcal H$ of
Lemma~\ref{lem:timeavg},
\[
\sup_{0\le t\le L_\varepsilon}
|\mathrm{Rem}_{\vartheta}(t)|
\le
\frac{\varepsilon C_g}{c_1}B_E(\rho_E).
\]
\end{lemma}

\begin{proof}
Subtract the two slow equations, which are driven
by the same Brownian motion and have the same
initial condition.
The Taylor expansion in the fast variable gives
\[
g(X^*(y)+z,y)=F(y)+G(y)z+R_g(z,y).
\]
Splitting $z=\mu_z+(z-\mu_z)$ and integrating
up to $t\wedge\vartheta$ proves
\eqref{eq:slow_decomposition}.

Since $\vartheta\le\tau$, the stopped path remains
in the domain of the Taylor remainder estimate.
Thus $|R_g(z,y)|\le C_g|z|^2$.
Using $|z|^2\le V/c_1$ and
$\vartheta\le\tau$, we obtain
\[
\begin{aligned}
|\mathrm{Rem}_{\vartheta}(t)|
&\le
\varepsilon C_g
\int_0^{t\wedge\vartheta}|z(s)|^2\,ds\\
&\le
\frac{\varepsilon C_g}{c_1}
\int_0^{t\wedge\tau}V(s)\,ds
=
\frac{\varepsilon C_g}{c_1}J(t).
\end{aligned}
\]
The final claim follows from the definition of
$\mathcal H$.
\end{proof}

We first bound the systematic component of the slow-variable
error. This estimate follows from the bound on the forcing
$b(y)$ and the uniform control of the inverse fast stability
matrix $A(y)^{-1}$.

\begin{lemma}[Bound on systematic term]\label{lem:sys}
\[
\sup_{t \le T/\varepsilon}|\Delta y_{\mathrm{sys}}(t\wedge\tau)|
\le
\frac{T K_A M_{D_xg}}{\lambda_0}
\left(
\varepsilon M_C M_g
+
\frac{\sigma^2}{2}M_{C_2}M_h^2
\right).
\]
\end{lemma}

\begin{proof}
Since
\[
\mu_z(y)=-A(y)^{-1}b(y),
\]
Assumption~\ref{ass:NH} implies
\[
A(y)^{-1}
=
-\int_0^\infty e^{A(y)s}\,ds,
\]
hence
\[
\|A(y)^{-1}\|
\le
\int_0^\infty \|e^{A(y)s}\|\,ds
\le
K_A\int_0^\infty e^{-\lambda_0 s}\,ds
=
\frac{K_A}{\lambda_0}.
\]
Using also $\|b(y)\|\le \beta$, we obtain
\[
\|\mu_z(y)\|
\le
\frac{K_A\beta}{\lambda_0}.
\]

Therefore
\[
|\Delta y_{\mathrm{sys}}(t\wedge\tau)|
\le
\varepsilon\int_0^{t\wedge\tau}
|D_xg(X^*(y(s)),y(s))|
\,|\mu_z(y(s))|
\,ds
\le
\varepsilon t\,M_{D_xg}
\frac{K_A\beta}{\lambda_0}.
\]

Substituting \(t=T/\varepsilon\) and the expression for \(\beta\)
yields
\[
\sup_{t\le T/\varepsilon}
|\Delta y_{\mathrm{sys}}(t\wedge\tau)|
\le
\frac{T K_A M_{D_xg}}{\lambda_0}
\left(
\varepsilon M_C M_g
+
\frac{\sigma^2}{2}M_{C_2}M_h^2
\right).
\]
\end{proof}

To estimate the accumulated fast deviation, we introduce a
$C^2$ corrector whose leading generator term is $G(y)z$.

\begin{lemma}[Corrector for the integrated fast deviation]
\label{lem:phi}
Set
\[
G(y)=D_xg(X^*(y),y),
\qquad
F(y)=g(X^*(y),y).
\]
After reducing $R>0$ if necessary, define on
$\mathcal D_R$ the function
\[
\chi(z,y)
=
\frac{
g\bigl(X^*(y)+\lambda_0A(y)^{-1}z,y\bigr)
-F(y)
}{\lambda_0}.
\]
Then $\chi$ is $C^2$ and satisfies
\[
\chi(0,y)=0,
\qquad
D_z\chi(0,y)A(y)=G(y).
\]
Let
\[
C_{\chi,1}
=\sup_{\mathcal D_R}\|D\chi\|,
\qquad
C_{\chi,2}
=\sup_{\mathcal D_R}\|D^2\chi\|.
\]
These constants are finite, and
\[
|\chi(z,y)|\le C_{\chi,1}|z|,
\]
\[
|D_z\chi(z,y)A(y)z-G(y)z|
\le M_A C_{\chi,2}|z|^2.
\]
\end{lemma}

\begin{proof}
Assumption~\ref{ass:NH} implies that $A(y)$ is
invertible and
\[
\|A(y)^{-1}\|\le \frac{K_A}{\lambda_0}.
\]
Thus
$|\lambda_0A(y)^{-1}z|\le K_A|z|$.
By compactness of the critical graph, $R$ can be
chosen small enough that every transformed point
in the definition of $\chi$ lies in $U$.

The smoothness assumptions imply that $\chi$ is
$C^2$, with bounded first and second derivatives
on $\mathcal D_R$.
Direct differentiation gives
\[
D_z\chi(0,y)=G(y)A(y)^{-1},
\]
which proves the identity at $z=0$.

Since $\chi(0,y)=0$, the mean-value formula yields
$|\chi(z,y)|\le C_{\chi,1}|z|$.
Also,
\[
\|D_z\chi(z,y)-D_z\chi(0,y)\|
\le C_{\chi,2}|z|.
\]
Multiplying by $A(y)z$ and using
$\|A(y)\|\le M_A$ proves the final bound.
\end{proof}

It\^o's formula for this corrector separates the integrated
fast forcing into boundary, remainder, and martingale terms.
The next lemma bounds the remainder and diffusion coefficient.

\begin{lemma}[It\^o decomposition of the corrector]
\label{lem:modA}
Let $\chi$ be defined in Lemma~\ref{lem:phi},
and suppose $0<\varepsilon\le1$.
Up to $\tau$, it satisfies
\[
d\chi(z(t),y(t))
=
\bigl[G(y(t))z(t)+\mathcal R_\chi(z(t),y(t))\bigr]dt
+\sigma K_\chi(z(t),y(t))\,dW_t,
\]
where
\[
K_\chi(z,y)
=
\bigl(D_y\chi(z,y)
-D_z\chi(z,y)DX^*(y)\bigr)h(y).
\]
There exists a constant $C_{\chi,R}$,
independent of $\varepsilon$ and $\sigma$,
such that on $\mathcal D_R$,
\[
|\mathcal R_\chi(z,y)|
\le C_{\chi,R}
\bigl(|z|^2+\varepsilon+\beta+\sigma^2\bigr).
\]
Moreover,
\[
\|K_\chi(z,y)\|_F
\le M_\chi,
\qquad
M_\chi
:=C_{\chi,1}\sqrt{1+M_C^2}\,M_h.
\]
The constants may depend on the fixed system
parameters and the chosen radius $R$.
\end{lemma}

\begin{proof}
Write $h_\ell(y)$ for the $\ell$th column of $h(y)$
and define the joint noise directions
\[
\nu_\ell(y)
=
\begin{pmatrix}
-DX^*(y)h_\ell(y)\\
h_\ell(y)
\end{pmatrix},
\qquad \ell=1,\dots,d.
\]
All derivatives of $\chi$ below are evaluated at $(z,y)$,
unless another evaluation point is displayed.

Applying It\^o's formula to the joint process $(z,y)$
gives the stated decomposition, with
\[
\begin{aligned}
\mathcal R_\chi(z,y)
={}&
\bigl(D_z\chi(z,y)-D_z\chi(0,y)\bigr)A(y)z\\
&+D_z\chi(z,y)\bigl(b(y)+N(z,y)\bigr)\\
&+\varepsilon D_y\chi(z,y)g(X^*(y)+z,y)\\
&+\frac{\sigma^2}{2}
\sum_{\ell=1}^d
D^2\chi(z,y)[\nu_\ell(y),\nu_\ell(y)].
\end{aligned}
\]
Here we used
$D_z\chi(0,y)A(y)z=G(y)z$.

By Lemma~\ref{lem:phi}, the first term is bounded by
$M_A C_{\chi,2}|z|^2$.
The bounds on $b$ and $N$ give
\[
|D_z\chi(b+N)|
\le C_{\chi,1}
\bigl(
\beta+\Gamma_1|z|
+\Gamma_2|z|^2+C_f|z|^3
\bigr).
\]
Since $|z|\le R$ and $0<\varepsilon\le1$,
\[
\Gamma_1|z|
\le \varepsilon M_C M_{D_xg}R,
\qquad
\Gamma_2
\le \frac{M_B}{2}+M_C C_g.
\]
Also,
\[
|\varepsilon D_y\chi\,g|
\le \varepsilon C_{\chi,1}M_g,
\]
and
\[
\sum_{\ell=1}^d|\nu_\ell(y)|^2
\le (1+M_C^2)M_h^2.
\]
These inequalities yield the claimed bound on
$\mathcal R_\chi$.

Finally, write
\[
K_\chi
=
D\chi
\begin{pmatrix}
-DX^*\\ I_m
\end{pmatrix}h.
\]
The bound on $D\chi$ and the Frobenius-norm convention
for $h$ imply
\[
\|K_\chi\|_F
\le C_{\chi,1}\sqrt{1+M_C^2}\,M_h.
\]
\end{proof}

Combining the corrector identity with the energy estimate
controls the fast contribution accumulated over $[0,T/\varepsilon]$.

\begin{lemma}[Bound on the stopped fast fluctuation term]
\label{lem:fast}
Let $L_\varepsilon=T/\varepsilon$,
$\rho_E\in(0,1)$, and
\[
\mathcal H
=\{J(L_\varepsilon)\le B_E(\rho_E)\},
\]
with $J$ and $B_E$ defined in Lemma~\ref{lem:timeavg}.
Recall that
\[
\mu_z(y)=-A(y)^{-1}b(y),
\]
and write
\[
\Delta y_{\mathrm{fast}}(t\wedge\tau)
=
\varepsilon\int_0^{t\wedge\tau}
G(y(s))\bigl(z(s)-\mu_z(y(s))\bigr)\,ds.
\]
Define the deterministic remainder bound
\begin{equation}\label{eq:fast_remainder}
\begin{aligned}
C_{\mathrm{rem}}
={}&2\varepsilon C_{\chi,1}R\\
&+C_{\chi,R}
\left[
\frac{\varepsilon}{c_1}B_E(\rho_E)
+T(\varepsilon+\beta+\sigma^2)
\right]\\
&+\frac{T K_A M_{D_xg}}{\lambda_0}\,\beta.
\end{aligned}
\end{equation}
For every $C_{\mathrm{fast}}>0$,
\begin{equation}\label{eq:fast_stopped}
\begin{aligned}
\PP\!\Bigl(
&\sup_{0\le t\le L_\varepsilon}
|\Delta y_{\mathrm{fast}}(t\wedge\tau)|\\
&>C_{\mathrm{rem}}
+C_{\mathrm{fast}}\sigma\sqrt{\varepsilon},
\quad \mathcal H
\Bigr)
\le p_{\mathrm{fast}},
\end{aligned}
\end{equation}
where, if $M_\chi>0$,
\begin{equation}\label{eq:pfast_corrector}
p_{\mathrm{fast}}
=
2m\exp\!\left(
-\frac{C_{\mathrm{fast}}^2}{2mT M_\chi^2}
\right).
\end{equation}
If $M_\chi=0$, one may take $p_{\mathrm{fast}}=0$.
Consequently, without restricting to $\mathcal H$,
the probability in \eqref{eq:fast_stopped}
is bounded by $\rho_E+p_{\mathrm{fast}}$.
\end{lemma}

\begin{proof}
Integrating the identity in Lemma~\ref{lem:modA}
up to $t\wedge\tau$ gives
\[
\begin{aligned}
\varepsilon\int_0^{t\wedge\tau}G(y(s))z(s)\,ds
={}&
\varepsilon\bigl[
\chi(z(t\wedge\tau),y(t\wedge\tau))
-\chi(z_0,y_0)
\bigr]\\
&-\varepsilon\int_0^{t\wedge\tau}
\mathcal R_\chi(z(s),y(s))\,ds\\
&-\varepsilon\sigma\int_0^{t\wedge\tau}
K_\chi(z(s),y(s))\,dW_s.
\end{aligned}
\]
The boundary term is bounded in norm by
$2\varepsilon C_{\chi,1}R$.

On $\mathcal H$, the comparison
$V\ge c_1|z|^2$ implies
\[
\int_0^{L_\varepsilon\wedge\tau}|z(s)|^2\,ds
\le \frac{B_E(\rho_E)}{c_1}.
\]
Using the remainder estimate in
Lemma~\ref{lem:modA}, we therefore obtain
\[
\begin{aligned}
&\sup_{t\le L_\varepsilon}
\left|
\varepsilon\int_0^{t\wedge\tau}
\mathcal R_\chi(z(s),y(s))\,ds
\right|\\
&\quad\le
C_{\chi,R}
\left[
\frac{\varepsilon}{c_1}B_E(\rho_E)
+T(\varepsilon+\beta+\sigma^2)
\right].
\end{aligned}
\]
Also, $\|\mu_z(y)\|\le K_A\beta/\lambda_0$, so
\[
\sup_{t\le L_\varepsilon}
\left|
\varepsilon\int_0^{t\wedge\tau}
G(y(s))\mu_z(y(s))\,ds
\right|
\le
\frac{T K_A M_{D_xg}}{\lambda_0}\,\beta.
\]
These three bounds give precisely
$C_{\mathrm{rem}}$.

It remains to estimate the continuous martingale
\[
N_\chi(t)
=
\varepsilon\sigma
\int_0^{t\wedge\tau}K_\chi(z(s),y(s))\,dW_s.
\]
For each component $j=1,\dots,m$,
\[
\langle N_{\chi,j}\rangle_{L_\varepsilon}
\le
\varepsilon^2\sigma^2M_\chi^2L_\varepsilon
=
T\varepsilon\sigma^2M_\chi^2.
\]
Suppose $\sigma M_\chi>0$.
The scalar exponential martingale inequality
(see \citet{RevuzYor1999}),
followed by a union bound over the $m$ components,
gives
\[
\begin{aligned}
\PP\!\left(
\sup_{t\le L_\varepsilon}|N_\chi(t)|
>C_{\mathrm{fast}}\sigma\sqrt{\varepsilon}
\right)
&\le
2m\exp\!\left(
-\frac{C_{\mathrm{fast}}^2}{2mT M_\chi^2}
\right).
\end{aligned}
\]
Here we used the fact that a vector of norm greater
than $a$ has a component of absolute value greater
than $a/\sqrt m$.
If $\sigma M_\chi=0$, then $N_\chi$ vanishes.

Combining the deterministic estimates on
$\mathcal H$ with this martingale bound proves
\eqref{eq:fast_stopped}.
Finally,
$\PP(\mathcal H^c)\le\rho_E$
by Lemma~\ref{lem:timeavg}, which gives the
unrestricted probability bound.
\end{proof}

\begin{remark}[Method]
The estimate uses an adapted corrector decomposition
and a stopped martingale inequality.
All stochastic integrands are adapted to the filtration
of the full system.
For background on stochastic averaging, see
\citet{Veretennikov1999}
and \citet[Ch.~17]{PavliotisStuart2008}.
\end{remark}

The remaining stochastic term comes from the difference
$h(y)-h(y_{\mathrm{ad}})$. Stopping at a candidate error level
gives a deterministic bound on its quadratic variation.

\begin{lemma}[Stopped bound on the martingale term]
\label{lem:mart}
Taking the Lipschitz constant $L_h$ in
Assumption~\ref{ass:globs} larger if necessary, write
\[
\|h(u)-h(v)\|_F\le L_h|u-v|.
\]
Let $\delta>0$ and define
\[
\tau_\delta
=\inf\{t\ge0:|\Delta y(t)|\ge\delta\},
\qquad
\theta_\delta=\tau\wedge\tau_\delta.
\]
Set
\[
M_\delta(t)
=
\sigma\int_0^{t\wedge\theta_\delta}
\bigl(h(y(s))-h(y_{\mathrm{ad}}(s))\bigr)\,dW_s.
\]
If $\sigma L_h>0$, then for every $a>0$,
\begin{equation}\label{eq:mart_stopped}
\PP\!\left(
\sup_{0\le t\le T/\varepsilon}|M_\delta(t)|\ge a
\right)
\le
2m\exp\!\left(
-\frac{a^2\varepsilon}
{2m\sigma^2L_h^2\delta^2T}
\right).
\end{equation}
If $\sigma L_h=0$, then $M_\delta$ vanishes.
\end{lemma}

\begin{proof}
By continuity and the definition of $\tau_\delta$,
\[
|\Delta y(s)|\le\delta
\qquad\text{for }s\le\theta_\delta.
\]
Consequently, for each component $j=1,\dots,m$,
\[
\begin{aligned}
\langle M_{\delta,j}\rangle_{T/\varepsilon}
&\le
\sigma^2
\int_0^{(T/\varepsilon)\wedge\theta_\delta}
\|h(y(s))-h(y_{\mathrm{ad}}(s))\|_F^2\,ds\\
&\le
\frac{\sigma^2L_h^2\delta^2T}{\varepsilon}.
\end{aligned}
\]
The scalar exponential martingale inequality
(see \citet{RevuzYor1999}) therefore gives
\[
\PP\!\left(
\sup_{t\le T/\varepsilon}|M_{\delta,j}(t)|
\ge\frac{a}{\sqrt m}
\right)
\le
2\exp\!\left(
-\frac{a^2\varepsilon}
{2m\sigma^2L_h^2\delta^2T}
\right).
\]
If a vector has norm at least $a$, one of its
components has absolute value at least $a/\sqrt m$.
A union bound over the $m$ components proves
\eqref{eq:mart_stopped}.
\end{proof}

We now combine the preceding estimates in a bootstrap
argument for $\Delta y$. The following lemma compares
the systematic, fluctuation, martingale, and remainder
contributions with the chosen error radius, accounting
for the amplification in Gr\"onwall's inequality.

\begin{lemma}[Stopped bootstrap for $\Delta y$]
\label{lem:bootstrap}
Under the preceding stopped estimates, fix
$\rho_E\in(0,1)$ and $C_{\mathrm{fast}}>0$.
Let $L_\varepsilon=T/\varepsilon$ and define
\begin{equation}\label{eq:slow_error_budget}
\begin{aligned}
a_\varepsilon
:={}&\varepsilon
+\frac{T K_A M_{D_xg}}{\lambda_0}\,\beta
+C_{\mathrm{rem}}\\
&+C_{\mathrm{fast}}\sigma\sqrt{\varepsilon}
+\frac{\varepsilon C_g}{c_1}B_E(\rho_E),
\end{aligned}
\end{equation}
where $C_{\mathrm{rem}}$ is defined in
\eqref{eq:fast_remainder}.
Set
\begin{equation}\label{eq:bootstrap_radius}
\delta=4e^{L_FT}a_\varepsilon.
\end{equation}
Then
\begin{equation}\label{eq:bootstrap_stopped}
\PP\!\left(
\sup_{0\le t\le L_\varepsilon}
|\Delta y(t\wedge\tau)|\ge\delta
\right)
\le \rho_E+p_{\mathrm{fast}}+p_{\mathrm{mart}},
\end{equation}
where $p_{\mathrm{fast}}$ is given by
\eqref{eq:pfast_corrector}, and, if $\sigma L_h>0$,
\begin{equation}\label{eq:pmart_corrected}
p_{\mathrm{mart}}
=
2m\exp\!\left(
-\frac{\varepsilon}
{32m\sigma^2L_h^2T e^{2L_FT}}
\right).
\end{equation}
If $\sigma L_h=0$, take $p_{\mathrm{mart}}=0$.
\end{lemma}

\begin{proof}
Let
\[
\tau_\delta
=\inf\{t\ge0:|\Delta y(t)|\ge\delta\},
\qquad
\theta_\delta=\tau\wedge\tau_\delta.
\]
Use the event $\mathcal H$ from
Lemma~\ref{lem:timeavg}, and define
\[
\begin{aligned}
\mathcal F
=\Bigl\{
&\sup_{t\le L_\varepsilon}
|\Delta y_{\mathrm{fast}}(t\wedge\tau)|\\
&\le C_{\mathrm{rem}}
+C_{\mathrm{fast}}\sigma\sqrt{\varepsilon}
\Bigr\},
\end{aligned}
\]
and
\[
\mathcal M
=
\left\{
\sup_{t\le L_\varepsilon}|M_\delta(t)|
\le a_\varepsilon
\right\},
\]
where $M_\delta$ is defined in
Lemma~\ref{lem:mart}.

On $\mathcal H\cap\mathcal F$, the systematic,
fast-fluctuation, and Taylor-remainder terms
in \eqref{eq:slow_decomposition}, stopped at
$\theta_\delta$, have total norm bounded by
$a_\varepsilon$.
Indeed, their respective bounds are
\[
\frac{T K_A M_{D_xg}}{\lambda_0}\,\beta,
\qquad
C_{\mathrm{rem}}
+C_{\mathrm{fast}}\sigma\sqrt{\varepsilon},
\qquad
\frac{\varepsilon C_g}{c_1}B_E(\rho_E).
\]
The same fast-fluctuation bound holds with
$\theta_\delta$ in place of $\tau$, because
$\theta_\delta\le\tau$.

On $\mathcal H\cap\mathcal F\cap\mathcal M$,
set
\[
u(t)=
\sup_{0\le s\le t}
|\Delta y(s\wedge\theta_\delta)|.
\]
The stopped decomposition and the global
Lipschitz bound for $F$ yield
\[
u(t)
\le 2a_\varepsilon
+\varepsilon L_F\int_0^t u(s)\,ds,
\qquad 0\le t\le L_\varepsilon.
\]
Gr\"onwall's inequality therefore gives
\[
u(L_\varepsilon)
\le 2a_\varepsilon e^{L_FT}
=\frac{\delta}{2}.
\]
By continuity, this excludes
$\tau_\delta\le L_\varepsilon\wedge\tau$.
Hence
\[
\left\{
\sup_{t\le L_\varepsilon}
|\Delta y(t\wedge\tau)|\ge\delta
\right\}
\subseteq
\mathcal H^c
\cup(\mathcal F^c\cap\mathcal H)
\cup\mathcal M^c.
\]

Lemma~\ref{lem:timeavg} gives
$\PP(\mathcal H^c)\le\rho_E$,
and Lemma~\ref{lem:fast} gives
$\PP(\mathcal F^c\cap\mathcal H)
\le p_{\mathrm{fast}}$.
Applying Lemma~\ref{lem:mart} with
$a=a_\varepsilon=\delta/(4e^{L_FT})$
gives
$\PP(\mathcal M^c)\le p_{\mathrm{mart}}$.
The union bound proves
\eqref{eq:bootstrap_stopped}.
\end{proof}

The bootstrap controls the error on the event where the
required concentration and fluctuation estimates hold.
We now collect the probabilities of the complementary
events to obtain the slow-variable error estimate.

\begin{theorem}[Adiabatic slow-variable error]
\label{thm:slow_error}
Under the conditions of Theorem~\ref{thm:main},
fix $p>0$ and $C_{\mathrm{fast}}>0$.
Set
\[
L_\varepsilon=T/\varepsilon,
\qquad
\rho=\rho_E=\varepsilon^p,
\]
and define $\delta$ by
\eqref{eq:slow_error_budget} and
\eqref{eq:bootstrap_radius}.
For fixed system parameters and sufficiently
small $\varepsilon$,
\begin{equation}\label{eq:slow_full_probability}
\begin{aligned}
\PP\!\Bigl(
&\sup_{0\le t\le L_\varepsilon}
|\Delta y(t)|\ge\delta
\ \text{or}\ \tau_K\le L_\varepsilon
\Bigr)\\
&\le
2\varepsilon^p+e^{-C_K}
+p_{\mathrm{fast}}+p_{\mathrm{mart}},
\end{aligned}
\end{equation}
where $p_{\mathrm{fast}}$ is as in
Lemma~\ref{lem:fast} and $p_{\mathrm{mart}}$
is given by \eqref{eq:pmart_corrected}.

For fixed geometric parameters,
$|z_0|=O(\varepsilon)$, and
$\sigma\le\sigma_c(\varepsilon)$, the radius satisfies
\[
\delta=O(\varepsilon).
\]
In particular, there exists a constant $C'>0$,
independent of $\varepsilon$ and $\sigma$,
such that
\[
\delta
\le
C'\bigl(
\varepsilon+\sigma\sqrt{\varepsilon}
+\sigma^2M_{C_2}
\bigr).
\]
The constant may depend on the fixed system
parameters, $T$, $p$, the chosen extension,
and $C_{\mathrm{fast}}$.
\end{theorem}

\begin{proof}
For sufficiently small $\varepsilon$,
$r_\rho<R$ by the argument proving
Theorem~\ref{thm:main}(a).
Lemma~\ref{lem:concentration} then gives
\[
\PP(\tau\le L_\varepsilon)
\le \rho+e^{-C_K}.
\]
Indeed, on
\[
\left\{
\sup_{t\le L_\varepsilon}|z(t)|\le r_\rho,
\quad \tau_K>L_\varepsilon
\right\},
\]
neither component of $\tau=\tau_R\wedge\tau_K$
can occur by time $L_\varepsilon$.

If $\tau>L_\varepsilon$, the stopped and
unstopped slow-variable errors coincide throughout
$[0,L_\varepsilon]$.
Therefore,
\[
\begin{aligned}
&\left\{
\sup_{t\le L_\varepsilon}|\Delta y(t)|\ge\delta
\right\}
\cup\{\tau_K\le L_\varepsilon\}\\
&\quad\subseteq
\{\tau\le L_\varepsilon\}
\cup
\left\{
\sup_{t\le L_\varepsilon}
|\Delta y(t\wedge\tau)|\ge\delta
\right\}.
\end{aligned}
\]
The union bound and
Lemma~\ref{lem:bootstrap} give
\[
\begin{aligned}
\PP\!\Bigl(
&\sup_{t\le L_\varepsilon}|\Delta y(t)|\ge\delta
\ \text{or}\ \tau_K\le L_\varepsilon
\Bigr)\\
&\le
\rho+e^{-C_K}+\rho_E
+p_{\mathrm{fast}}+p_{\mathrm{mart}},
\end{aligned}
\]
which proves \eqref{eq:slow_full_probability}.

For the radius estimate, fixed system parameters
and $\sigma^2=O(\varepsilon)$ imply
\[
\beta=O(\varepsilon),
\quad
D_0=O(\varepsilon),
\quad
\eta^2=O(\varepsilon),
\quad
\gamma_V\ge\gamma_*>0.
\]
Since $V(0)=O(\varepsilon^2)$ and
$\rho_E=\varepsilon^p$,
\eqref{eq:energy_budget} yields
\[
B_E(\rho_E)=O(1).
\]
Consequently,
\eqref{eq:fast_remainder} gives
$C_{\mathrm{rem}}=O(\varepsilon)$.
Also $\sigma\sqrt{\varepsilon}=O(\varepsilon)$.
Equations~\eqref{eq:slow_error_budget}
and~\eqref{eq:bootstrap_radius} now imply
$\delta=O(\varepsilon)$,
and hence the stated upper bound.
\end{proof}

\bigskip\noindent\textbf{Proof of Theorem~\ref{thm:main}}.

For part~(a), fix $p>0$ and set
$\rho=\varepsilon^p$.
For fixed system parameters,
$|z_0|=O(\varepsilon)$ and
$\sigma\le\sigma_c(\varepsilon)\le C_0\sqrt{\varepsilon}$
imply
\[
D_0=O(\varepsilon^2+\sigma^2+\sigma^4),
\qquad
\eta^2=O(\sigma^2).
\]
Moreover, $\tilde\lambda\to1$ as
$\varepsilon\to0$ in this regime.
Hence there exists a constant $\gamma_*>0$
such that $\gamma_V\ge\gamma_*$
for sufficiently small $\varepsilon$.

Using \eqref{eq:confidence_radius}, we obtain
\[
r_\rho
=O\!\left(
\varepsilon+\sigma^2
+\sigma\sqrt{\log(1/\varepsilon)}
\right).
\]
In particular,
\[
r_\rho^2
=O\!\left(\varepsilon\log(1/\varepsilon)\right)
\longrightarrow0.
\]
Thus $r_\rho<R$ for sufficiently small
$\varepsilon$, with $R$ fixed and admissible
in \eqref{eq:Rcond}.
Applying \eqref{eq:confidence_joint} gives
\eqref{eq:path_conc}.

Part~(b) follows from Theorem~\ref{thm:slow_error},
choosing $C'$ so that $\delta\le\delta_\ast$. \hfill\qedsymbol

\section{Discussion}
\label{sec:discussion}

\subsection{\texorpdfstring{Phase diagram in $(\varepsilon,\sigma)$ space}{Phase diagram in the epsilon-sigma plane}}

The sufficient regime ensuring path concentration in the $(\varepsilon,\sigma)$ plane is described by the two branches
of $\sigma_c(\varepsilon)$ in \eqref{eq:sigmac}. Figure~1
shows a schematic phase diagram.

\begin{figure}[ht]
\centering
\begin{tikzpicture}[>=latex]
  
  \def\xmax{6} \def\ymax{3.25}

  \foreach \x in {0,1,2,3,4,5,6}
    {\draw[gray!15, thin] (\x,0) -- (\x,\ymax);}
  \foreach \y in {0,0.5,1,1.5,2,2.5,3}
    {\draw[gray!15, thin] (0,\y) -- (\xmax,\y);}

  \fill[blue!14] (0,0) -- (0,1) -- (4,2) -- cycle;

  \draw[->, thick] (-0.05,0) -- (\xmax,0)
    node[right] {$-\log\varepsilon$};
  \draw[->, thick] (0,-0.05) -- (0,\ymax)
    node[above] {$-\log\sigma$};

  \foreach \x in {1,2,3,4,5,6}
    {\draw[thick] (\x,0.05) -- (\x,-0.05)
       node[below, font=\footnotesize] {\x};}

  \foreach \y in {0.5,1,1.5,2,2.5,3}
    {\draw[thick] (0.05,\y) -- (-0.05,\y)
       node[left, font=\footnotesize] {\y};}

  \draw[blue, dashed, very thick] (0,0) -- (6,3);

  \draw[red, solid, very thick] (0,1) -- (6,2.5);

  \filldraw[black] (4,2) circle (1.8pt);
\end{tikzpicture}
\caption{Phase diagram of the sufficient concentration regime given by $\sigma_c(\varepsilon)$ in $(-\log\varepsilon,-\log\sigma)$-space, drawn for a fixed manifold geometry (i.e., fixed constants $C_0$ and $L_{\mathrm{geom}}$). 
\textbf{Blue dashed}: classical threshold $-\log\sigma = \tfrac12(-\log\varepsilon) - \log C_0$.
\textbf{Red solid}: curvature-dependent threshold $-\log\sigma = \tfrac14(-\log\varepsilon) - \log(C_0 L_{\mathrm{geom}}/\sqrt{\lambda_0})$.
\textbf{Shaded}: regime where curvature provides a strictly tighter upper bound on the allowable noise. (Higher values on the vertical axis represent strictly smaller noise $\sigma$.)}
\label{fig:phase}
\end{figure}
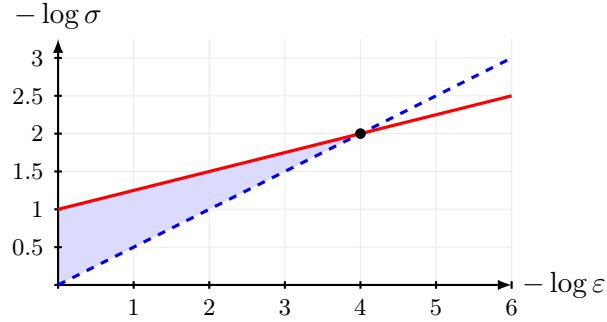

\subsection{When does curvature matter?}
For a single fixed manifold ($M_{C_2}$, $M_h$ constant), the
curvature branch of $\sigma_c$ is active when
\[
\varepsilon^{1/4}\frac{L_{\mathrm{geom}}}{\sqrt{\lambda_0}}
< \sqrt{\varepsilon}
\;\;\Longleftrightarrow\;\;
M_{C_2}\,M_h^2 > \varepsilon^{-1/2}.
\]
This is difficult to realise for a single manifold, since
$\varepsilon^{-1/2}$ is large. The curvature threshold is more
naturally activated in a family of manifolds parametrised
by a geometric shape parameter $\delta \to 0$ for which
$M_{C_2}(\delta) \to \infty$, while the stability index
$\lambda_0$ and the Lyapunov constants $c_1$, $c_2$, $M_{DP}$,
$M_{D^2P}$ remain bounded. This is the two-parameter regime
analysed in Example~\ref{ex:sigmoid}.

This is complementary to the theory of \citet{BerglundGentz2003,BerglundGentz2006}, 
which captures the regime $\lambda_0 \to 0$ near fold bifurcations
(where $\sigma_c \sim \varepsilon^{1/4}$ but the mechanism is
the vanishing of the spectral gap, not curvature). Our analysis
captures the regime $\|D^2X^*\| \to \infty$ with $\lambda_0 > 0$
fixed, in the joint limit $(\varepsilon, \delta) \to (0,0)$ with
$\varepsilon^{1/2} \ll \delta < \varepsilon^{1/4}$.

While the curvature-dominated regime requires a parametrically
large second derivative $\|D^2X^*\|$, such conditions arise
naturally in systems with sharp transition layers, sigmoidal
response functions, or highly nonlinear coupling between slow
and fast variables. Even when the curvature branch is not
strictly binding, the It\^o curvature forcing $b_\sigma$ satisfies
$|b_\sigma(y)|\leq \sigma^2 M_{C2}M_h^2/2$ and contributes to
the Lyapunov forcing coefficient $D_0$. Through $D_0$, this
geometric contribution enters the confidence radius $r_\rho$
in~\eqref{eq:path_conc}, quantifying the effect of manifold
geometry on the fast-variable spread.

\subsection{Concrete example: sigmoid slow manifold}
\label{sec:example}

We exhibit a family of systems, parametrised by a geometric
parameter $\delta \to 0$, for which $\|D^2X^*\|$ diverges while
$\lambda_0$ remains bounded away from zero. For each fixed member of the family, Theorem~\ref{thm:main} applies with an admissibility constant $C_0(\delta)$ determined by the corresponding system parameters. The analysis below tracks this dependence, identifies when the
curvature branch of the sufficient bound is the smaller one,
and verifies concentration in the joint limit
$(\varepsilon,\delta)\to(0,0)$.

\begin{example}[Steep sigmoid transition -- two-parameter analysis]
\label{ex:sigmoid}
Consider $x \in \RR, y \in \RR$ with
\[
dx = (-x + \tanh(y/\delta))\,dt, \qquad
dy = \varepsilon(-y)\,dt + \sigma\,dW_t,
\]
where $\delta \ll 1$ controls the sharpness of the transition.
The slow manifold is $X^*(y) = \tanh(y/\delta)$ and the fast
Jacobian is $A(y) = D_x f = -1$.
The Taylor remainders in the fast variable vanish for this model,
so the localization estimates can be checked with a fixed fast
radius along the joint regime considered below.

\medskip\noindent\textbf{Constants and their $\delta$-scaling.}
\begin{itemize}
\item $A(y) = -1$ is constant, so $\lambda_0 = 1$, $K_A = 1$,
      and $DA(y) = D^2A(y) \equiv 0$. Since $A$ does not depend
      on $y$, the Lyapunov equation $A^T P + PA = -I$ has the
      explicit solution $P(y) = \tfrac12 I$, yielding
      $c_1 = c_2 = \tfrac12$. Crucially, $DP = D^2P = 0$, so
      $M_{DP} = M_{D^2P} = 0$ for all $\delta > 0$.
\item $X^{*\prime}(y) = \delta^{-1}\operatorname{sech}^2(y/\delta)$,
      so $M_C = 1/\delta$.
      $X^{*\prime\prime}(y) = -2\delta^{-2}\tanh(y/\delta)\operatorname{sech}^2(y/\delta)$,
      so $M_{C_2} = \Theta(1/\delta^2)$ (with the maximum attained when $\tanh(y/\delta)=\pm1/\sqrt{3}$).
\item $g(x,y) = -y$, so $M_g = R$, $M_{D_xg} = 0$. $h(y) = 1$, so $M_h = 1$.
\item Compact domain $K=[-R,R]$. With $\psi(y)=y^2$ and a
      fixed safe core $[-a,a]$, $0<a<R$, the drift is
      $\psi'(y)g(x,y)=-2y^2\le-2a^2$ outside that core,
      uniformly in $x$ and $\delta$.
\end{itemize}

\medskip\noindent\textbf{Geometric scale and critical noise.}
$L_{\mathrm{geom}} = \sqrt{\lambda_0/(M_{C_2}M_h^2)} = \Theta(\delta)$.
Before inserting the common prefactor $C_0(\delta)$, the two factors scale as
\[
\sigma_c^{\mathrm{curv}}(\varepsilon) \sim \varepsilon^{1/4}\delta, \qquad
\sigma_c^{\mathrm{class}}(\varepsilon) \sim \sqrt{\varepsilon}.
\]
The curvature branch is strictly tighter when
$\varepsilon^{1/4}\delta < \sqrt{\varepsilon}$, i.e.\
$\delta < \varepsilon^{1/4}$.

\medskip\noindent\textbf{Branch-dependent sufficient thresholds.}
The common prefactor must also account for $M_C=\delta^{-1}$.
For this family, choose
\[
C_0(\delta)=c\delta,
\]
with a sufficiently small fixed $c>0$.
Indeed, the diffusion coefficient in
\eqref{eq:threshold_energy_control} is
$a_0=c_2M_C^2M_h^2=1/(2\delta^2)$, so
$a_0C_0(\delta)^2=c^2/2$ remains uniformly bounded.
For $\varepsilon/\delta^2\to0$, the remaining deterministic
term in that estimate is controlled as well.
Substitution into~\eqref{eq:sigmac} then gives
\[
\sigma\lesssim
\min\!\bigl(\delta\sqrt{\varepsilon},
             \delta^2\varepsilon^{1/4}\bigr).
\]
These sufficient branches use the same common prefactor.
Within this sufficient bound, the curvature branch is the
smaller one when $\delta<\varepsilon^{1/4}$.

\medskip\noindent\textbf{Self-consistency at the curvature threshold.}
Set $\sigma = \sigma_c^{\mathrm{curv}} = C\delta^2\varepsilon^{1/4}$
for an appropriate constant $C > 0$. We verify $D_0/\varepsilon = O(1)$.

Since $M_{DP} = 0$, the cross-It\^o correction in $D_0$ vanishes
and $\tilde\lambda = 1$. Thus
$D_0 = \sigma^2 c_2 M_C^2 M_h^2 + 4(c_2\beta)^2$
with $\beta = \varepsilon M_C M_g + \tfrac{\sigma^2}{2}M_{C_2}M_h^2
= \Theta(\varepsilon/\delta + \delta^2\varepsilon^{1/2})$.

The diffusion-energy term:
$\sigma^2 c_2 M_C^2 M_h^2/\varepsilon = \Theta(\delta^4\varepsilon^{1/2}\cdot\delta^{-2}/\varepsilon)
= \Theta(\delta^2\varepsilon^{-1/2})$. This is $O(1)$ iff
$\delta \le \varepsilon^{1/4}$, which holds in the curvature-active
regime.

The curvature-bias term:
$\beta^2/\varepsilon = \Theta(\varepsilon/\delta^2 + \varepsilon^{1/2}\delta + \delta^4)$.
This vanishes as $(\varepsilon,\delta) \to 0$ provided
$\varepsilon/\delta^2 \to 0$, i.e.\ $\delta \gg \sqrt{\varepsilon}$.

The curvature branch is therefore self-consistent in the
joint regime
\[
\sqrt{\varepsilon} \ll \delta < \varepsilon^{1/4},
\]
which is non-empty since $\varepsilon^{1/2} < \varepsilon^{1/4}$ for
all $\varepsilon < 1$.

\medskip\noindent\textbf{Confidence on the long interval.}
Here $\gamma_V=1$, $D_0=\sigma^2/(2\delta^2)+\beta^2$,
and $\eta^2=4\sigma^2/\delta^2$.
For $\rho=\varepsilon^p$, equation~\eqref{eq:confidence_radius}
therefore reads
\[
\begin{aligned}
r_\rho^2
={}&|z_0|^2+4\beta^2\\
&+\frac{\sigma^2}{\delta^2}
\left[2+8\log\!\left(
\frac{1+T/\varepsilon}{\varepsilon^p}\right)\right].
\end{aligned}
\]
At $\sigma=C\delta^2\varepsilon^{1/4}$ this yields
\[
r_\rho
=O\!\left(
\frac{\varepsilon}{\delta}
+\delta^2\sqrt{\varepsilon}
+\delta\varepsilon^{1/4}\sqrt{\log(1/\varepsilon)}
\right)
\]
for a uniform initial bound $|z_0|=O(\varepsilon)$.
Consequently, $r_\rho\to0$ along the same joint regime
$\sqrt{\varepsilon}\ll\delta<\varepsilon^{1/4}$.
The fast-variable failure bound is $\varepsilon^p$,
with the slow exit term controlled separately by
\eqref{eq:exit_finite_horizon}.
For a fixed compact slow domain its confinement constants
are independent of $\delta$ because $g(x,y)=-y$ and $h=1$.
At this noise scale,
$\varepsilon/\sigma^2=C^{-2}\sqrt{\varepsilon}/\delta^4
\to\infty$ in the stated regime, so that exit bound also
vanishes. Thus the same family admits a nontrivial
long-time concentration estimate with diverging curvature.

\medskip\noindent\textbf{Numerical regime.}
For $\varepsilon = 10^{-8}$ and $\delta = 0.01$:
$\varepsilon^{1/4} = 10^{-2} = \delta$, so the two thresholds are
comparable ($\sigma_c^{\mathrm{class}} \sim \sigma_c^{\mathrm{curv}}$).
The curvature branch is strictly dominant for
$\delta \ll \varepsilon^{1/4}$. For example, with $\varepsilon = 10^{-12}$
and $\delta = 10^{-4}$: $\varepsilon^{1/4} = 10^{-3} \gg \delta$,
$\sqrt{\varepsilon} = 10^{-6} \ll \delta$, confirming
$\delta$ lies in the valid joint regime. Such steep sigmoidal
transitions arise in gene-regulatory switches, kinase cascades,
and strongly saturating chemical reactions.
For this particular family, $g(x,y)=-y$ is independent of $x$,
so the full and adiabatic slow variables coincide; the example
therefore tests the fast-variable concentration estimate.
\end{example}

\paragraph{Scope of Example~\ref{ex:sigmoid} and the non-constant $A(y)$ case.}
Example~\ref{ex:sigmoid} uses $A(y)\equiv-1$, so
$M_{DP}=M_{D^2P}=0$. This isolates the curvature mechanism
while keeping the spectral gap fixed.
The general theorem also allows non-constant $A(y)$:
the derivative bounds in Appendix~\ref{sec:appendix} enter
the drift and quadratic-variation estimates explicitly.
For a parameter-dependent family, its joint concentration
regime is determined by inserting these bounds into
\eqref{eq:Rcond}, the effective decay rate,
the localization condition $r_\rho<R$, and the energy budget
\eqref{eq:energy_budget}.
The sigmoid calculation gives a non-empty such regime with
an active curvature branch.
This regime complements the near-bifurcation theory of
\citet{BerglundGentz2003,BerglundGentz2006}:
their analysis tracks $\lambda_0\to0$, whereas the present
example has $\|D^2X^*\|\to\infty$ with $\lambda_0>0$ fixed.

\begin{remark}[Hill-type regulatory functions]
The sigmoid $X^*(y) = \tanh(y/\delta)$ in Example~\ref{ex:sigmoid}
is related, after rescaling, to steep Hill-type regulatory profiles
$H(y) = y^n/(K^n + y^n)$ as the cooperativity $n \to \infty$. Such
switch-like behaviour appears in gene regulation, kinase cascades,
and synthetic biology. For a Hill function with fixed half-saturation $K>0$,
the maximum second-derivative magnitude satisfies
$\sup_{y>0}|H^{\prime\prime}(y)|=\Theta(n^2/K^2)$ as $n\to\infty$.
When the domain contains the transition layer, this gives
$M_{C_2}=\Theta(n^2)$ and $L_{\mathrm{geom}}=\Theta(1/n)$. Before the common prefactor $C_0(n)$ is included, the curvature
factor scales as $\varepsilon^{1/4}/n$, whereas
$\sigma_c^{\mathrm{class}} \sim \varepsilon^{1/2}$; the curvature
branch is active when $n \gg \varepsilon^{-1/4}$. Although this
requires a very high cooperativity (or a very steep transition
layer in the limit), the example demonstrates a mathematically
well-defined regime where the geometrically refined fast-variable
bound is active, with the parameter dependence of the common prefactor and
all other bounds checked as in the sigmoid example.
\end{remark}

\subsection{Physical motivation}
Systems with noise only on the slow variable naturally arise
when:
\begin{itemize}
\item A fast chemical reaction equilibrates rapidly, while
a slow environmental parameter (temperature, pH) undergoes
stochastic fluctuations.
\item Neural or biophysical models with deterministic fast
gating kinetics and stochastically driven slow recovery
variables \citep{NewbyBressloffKeener2013}. Recent work by
\citet{Wu2026} on excitable fast--slow systems with 
multiplicative Feller noise on the slow variable identifies the 
practical boundary where the adiabatic approximation breaks down
numerically.
\item Climate models where fast atmospheric dynamics are
slaved to slowly varying, stochastically perturbed oceanic
or ice-sheet parameters.
\end{itemize}

\subsection{Practical value of geometry-dependent fast-variable bounds}
\label{sec:pr-value}
While Theorem~\ref{thm:main}(b) gives an $O(\varepsilon)$
slow-variable error bound for fixed geometry, the curvature-refined
fast-variable bound in part~(a) has independent practical value in
several contexts. First, when the fast variable itself is the
observable of interest, such as fast ionic currents in
neural gating models or rapidly equilibrating concentrations in
enzyme-kinetics, an explicit track of $x$ around $X^*(y)$ directly
limits measurement-prediction error. Second, in feedback control
of multi-scale systems, the actuator may depend on the
fast-variable state; the geometric bound could inform stability
margins that account for the curvature forcing. Third, in numerical methods
for stiff SDEs, adaptive time-stepping schemes could use the explicit
confidence radius when allocating pathwise error budgets. Finally, the curvature correction may be useful whenever the effective slow manifold exhibits large local second derivatives, regardless of whether these arise from analytic structure or from numerical approximation. The geometric correction thus refines the concentration tube around the slow manifold itself.

\subsection{Limitations and extensions}
The principal limitations of the present work are:

\textbf{(a)} We assumed $D_xh\equiv0$.
For $h=h(x,y)$, both the transmitted diffusion
$-\sigma DX^*(y)h(X^*(y)+z,y)$ and the curvature contraction
in the exact deviation drift depend on $z$.
Extending the results requires recomputing the stopped drift,
quadratic-variation and corrector estimates with those coefficients.

\textbf{(b)} The assumption $f \in C^4$ is required only for
the bound on $D^2P(y)$ (Appendix~\ref{sec:appendix}). A weaker
result using only $DP$ (and hence $f \in C^3$) can be obtained
at the cost of a slightly weaker exponential rate.

\textbf{(c)} Assumption~\ref{ass:confinement} requires
uniform inward drift across the level sets of $\psi$
on $K\setminus\mathrm{int}(K_0)$, for all fast states
in the prescribed neighbourhood of $M_0$.
The resulting finite-horizon estimate is given
in~\eqref{eq:exit_finite_horizon}.
Its uniform form $e^{-C_K}$ uses a sufficiently small
prefactor $C_0$, depending on the time horizon and
the confinement data.
Enlarging $K$ also changes the uniform coefficient bounds,
so its effect must be assessed through all the quantities
appearing in that estimate.

\textbf{(d)} Near bifurcation points ($\lambda_0 \to 0$),
the upper Lyapunov bound and derivative estimates need not
remain uniform, and the present argument does not ensure a
uniform Foster--Lyapunov gap. The near-bifurcation regime
requires the Berglund--Gentz sample-paths technique, which
tracks $y$ through the bifurcation and uses local normal-form
coordinates. Combining both approaches into a single global result,
capable of simultaneously handling large curvature and
vanishing spectral gap, remains an interesting open problem.

\appendix
\section{Bounds on the parameter-dependent Lyapunov matrix}
\label{sec:appendix}

\begin{lemma}\label{lem:DP}
Under Assumption~\ref{ass:NH}, with $A \in C^2(K)$:
\begin{equation}
\|DP(y)\| \le \frac{K_A^4\|DA\|_K}{2\lambda_0^2} =: M_{DP},
\end{equation}
where $DA(y) = D_y[D_xf(X^*(y),y)]$.
\end{lemma}

\begin{proof}
Differentiate $A^T P + PA = -I$ with respect to $y_k$:
$(D_kA)^T P + A^T(D_kP) + (D_kP)A + P(D_kA) = 0$.
This is the Sylvester equation
$\mathcal{L}_{A^T,A}(D_kP) = -[(D_kA)^TP + P(D_kA)]$. The solution of $A^T X + XA = Y$ is $X = -\int_0^\infty e^{A^T s} Y e^{As}\,ds$,
so taking operator norms:
$\|X\| \le \int_0^\infty \|e^{As}\|^2 \|Y\|\,ds \le K_A^2 \|Y\| \int_0^\infty e^{-2\lambda_0 s}\,ds = \frac{K_A^2}{2\lambda_0}\|Y\|$.
Hence the Sylvester operator inverse satisfies $\|\mathcal{L}_{A^T,A}^{-1}\| \le K_A^2/(2\lambda_0)$.
The right-hand side has norm $\le 2 \|D_kA\| \|P\| \le 2 \|D_kA\| c_2$
where $c_2 = K_A^2/(2\lambda_0)$. Therefore
\[
\|D_kP\| \le \frac{K_A^2}{2\lambda_0} \cdot 2 \|D_kA\| \cdot \frac{K_A^2}{2\lambda_0}
= \frac{K_A^4 \|D_kA\|}{2\lambda_0^2},
\]
giving $M_{DP} = K_A^4 \|DA\|_K / (2\lambda_0^2)$.
Taking arbitrary unit directions instead of coordinate directions
gives the stated induced derivative-norm bound.
\end{proof}

\begin{lemma}\label{lem:D2P}
Under Assumptions~\ref{ass:NH} and~\ref{ass:smooth}, we have $P \in C^2(K)$ and
\begin{equation}
\|D^2P(y)\| \le \frac{K_A^4}{2\lambda_0^2}\left(\|D^2A\|_K
+ \frac{2K_A^2\|DA\|_K^2}{\lambda_0}\right) =: M_{D^2P}.
\end{equation}
\end{lemma}

\begin{proof}
Differentiate twice. The second derivative satisfies
$\mathcal{L}_{A^T,A}(D^2_{jk}P) = -R_{jk}$ where the right-hand
side $R_{jk}$ contains four terms involving $(D_jA)(D_kP) + (D_kA)(D_jP)$
plus two terms $(D^2_{jk}A)^T P + P(D^2_{jk}A)$.
Bounding: $\|R_{jk}\| \le 4\|DA\| M_{DP} + 2c_2\|D^2A\|$.
Substituting $M_{DP} = K_A^4 \|DA\|_K / (2\lambda_0^2)$ and
$c_2 = K_A^2/(2\lambda_0)$:
\[
\|R_{jk}\| \le \frac{2K_A^4 \|DA\|^2}{\lambda_0^2} + \frac{K_A^2 \|D^2A\|}{\lambda_0}.
\]
Applying $\|\mathcal{L}^{-1}\| \le K_A^2/(2\lambda_0)$:
\[
\|D^2_{jk}P\| \le \frac{K_A^2}{2\lambda_0}\left(\frac{2K_A^4 \|DA\|^2}{\lambda_0^2} + \frac{K_A^2 \|D^2A\|}{\lambda_0}\right)
= \frac{K_A^6 \|DA\|^2}{\lambda_0^3} + \frac{K_A^4 \|D^2A\|}{2\lambda_0^2}.
\]
Factoring out $K_A^4/(2\lambda_0^2)$ gives the stated result.
The same calculation with directional derivatives in arbitrary unit
directions gives the stated induced tensor-norm bound.
\end{proof}

\bibliographystyle{abbrvnat}
\bibliography{references}

\end{document}